\documentclass{amsart}
\usepackage[latin1]{inputenc}
\usepackage[T1]{fontenc}
\usepackage{euscript}
\usepackage{amssymb, amsmath, color, textcomp, url,tikz, mathtools}
\usetikzlibrary{matrix,arrows}
\usetikzlibrary{fit}
\usetikzlibrary{positioning}
\usepackage[labelformat=empty]{caption}

\usepackage[normalem]{ulem}
\usepackage{soul}
\usepackage{mdwlist}

\usepackage{tikz-cd}

\RequirePackage{ifpdf}
\ifpdf
\usepackage[pdftex]{hyperref}
\else
\usepackage[hypertex]{hyperref}
\fi

\usepackage{enumerate,xspace}
\usepackage{chngcntr, csquotes} 
\theoremstyle{plain}
\newtheorem{theorem}{Theorem}[section]
\newtheorem{cor}[theorem]{Corollary}
\newtheorem{prop}[theorem]{Proposition}
\newtheorem{lemma}[theorem]{Lemma}

\newcounter{proofcount}
\AtBeginEnvironment{proof}{\stepcounter{proofcount}}

\newtheorem*{claim*}{Claim}
\makeatletter                  
\@addtoreset{claim}{proofcount}
\makeatother                   

\usepackage{pdfpages}

\newenvironment{claimproof*}[1][Proof of Claim.] 
{%
	\proof[#1]%
	
}
{%
	\endproof%
}

\newtheorem{thm}{Theorem}

\theoremstyle{definition}
\newtheorem{remark}[theorem]{Remark}
\newtheorem{fact}[theorem]{Fact}
\newtheorem{definition}[theorem]{Definition}
\newtheorem{example}[theorem]{Example}

\newcommand{\nc}{\newcommand}

\nc{\Z}{\mathbb{Z}}
\nc{\Q}{\mathbb{Q}}
\nc{\N}{\mathbb{N}}
\nc{\F}{\mathbb{F}}
\nc{\UU}{\mathbb{U}}
\nc{\C}{\mathbb{C}}
\nc{\R}{\mathbb{R}}
\nc{\Pp}{\EuScript{P}}
\nc{\Ff}{\EuScript{F}}
\nc{\Kk}{\EuScript{K}}

\nc{\m}{\mathfrak{m}}
\nc{\g}{\mathfrak{g}}
\nc{\lr}{\mathfrak{l}}
\nc{\z}{\mathfrak{z}}
\nc{\h}{\mathfrak{h}}
\nc{\A}{\mathfrak{a}}
\nc{\B}{\mathfrak{b}}

\nc\II{\mathcal I}

\nc{\stt}{\operatorname{St}}
\nc{\stab}{\operatorname{Stab}}
\nc{\GO}[1]{G_{#1}^{00}}
\nc{\sbgp}[1]{\langle\xspace {#1}\xspace\rangle}
\nc{\Conn}[1]{\langle\xspace {X}\xspace\rangle^{00}_{#1}}
\nc{\band}[1]{\bar d_{\mathcal{#1}}}
\nc\Def{\operatorname{Def}}

\nc{\ad}{\operatorname{ad}}
\nc{\dcl}{\operatorname{dcl}}
\nc{\acl}{\operatorname{acl}}
\nc{\cb}{\operatorname{Cb}}
\nc{\tp}{\operatorname{tp}}
\nc{\stp}{\operatorname{stp}}
\nc{\ann}{\operatorname{ann}}
\nc{\Aut}{\operatorname{Aut}}
\nc{\im}{\operatorname{im}}
\nc\inv{ ^{-1}}

\nc\U{\operatorname{U}}
\nc{\cf}{\text{cf.\,}}
\nc{\eg}{\text{e.g. }}

\def\Ind#1#2{#1\setbox0=\hbox{$#1x$}\kern\wd0\hbox to
	0pt{\hss$#1\mid$\hss} \lower.9\ht0\hbox to
	0pt{\hss$#1\smile$\hss}\kern\wd0}
\def\Notind#1#2{#1\setbox0=\hbox{$#1x$}\kern\wd0\hbox to
	0pt{\mathchardef\nn="0236\hss$#1\nn$\kern1.4\wd0\hss}\hbox to
	0pt{\hss$#1\mid$\hss}\lower.9\ht0 \hbox to
	0pt{\hss$#1\smile$\hss}\kern\wd0}

\def\indip{\mathop{\ \ \hbox to 0pt{\hss$\mid^{\hbox to
				0pt{$\scriptstyle P$\hss}}$\hss}
		\lower4pt\hbox to 0pt{\hss$\smile$\hss}\ \ }}
\def\nindip{\mathop{\ \ \hbox to 0pt{\hss$\!\not{\mid}^{\hbox to
				0pt{$\scriptstyle\, P$\hss}}$\hss}
		\lower4pt\hbox to 0pt{\hss$\smile$\hss}\ \ }}

\title[Rings and the CBP]{Rings of finite Morley rank without the canonical base property}

\address{ \, Abteilung f\"ur Mathematische Logik, Mathematisches Institut,
	Albert-Ludwigs-Universit\"at Freiburg, Ernst-Zermelo-Stra\ss e 1, D-79104
	Freiburg, Germany}
\email{1michaelloesch@gmail.com}

\address{Departamento de \'Algebra, Geometr\'ia y Topolog\'ia; Facultad de Matem\'aticas;
	Universidad Complutense de Madrid; 28040 Madrid, Spain}

\email{dpalacin@ucm.es}

\date{\today}

\author{Michael Loesch and Daniel Palac\'in}

\thanks{The second author is supported by Spanish STRANO PID2021-122752NB-I00 and Grupos UCM 910444,  also was partially supported by the contract 2020-T1/TIC-20313 from Community of Madrid. 
}

\begin{document}
	
	\begin{abstract} 
	We present numerous natural algebraic examples without the so-called Canonical Base Property (CBP). 
		We prove that every commutative unitary ring of finite Morley rank without finite-index proper ideals satisfies the CBP if and only if it is a field, a ring of positive characteristic or a finite direct product of these. In addition, we construct a CM-trivial commutative local ring with a finite residue field without the CBP. Furthermore, we also show that finite-dimensional non-associative algebras over an algebraically closed field of characteristic $0$ give rise to triangular rings without the CBP. This also applies to Baudisch's $2$-step nilpotent Lie algebras, which yields the existence of a $2$-step nilpotent group of finite Morley rank whose theory, in the pure language of groups, is CM-trivial and does not satisfy the CBP. 
	\end{abstract}
	
\maketitle
	
	\section{Introduction}
	
	The interplay between model theory and algebraic structures has been in the spotlight of model theorist since the contributions of Macintyre \cite{aM71}, who proved that every uncountably categorical infinite field is algebraically closed and hence (almost) strongly minimal. Afterwards, during the decade of the 1970's some results on rings were obtained through the lens of categoricity theory, and more generally of Shelah's stability theory. We refer to \cite{CR76,aC77} for results on commutative unitary rings with an infinite-index Jacobson ideal, as well as to \cite[Appendix A.1]{BN94} for some further results.
	
	
	In contrast to Macintyre's result, Podewski and Reineke \cite{PR74} proved that not every uncountably categorical ring is almost strongly minimal. They showed that the theory of the local ring $\C[x]/(x^2)$ is uncountably categorical but not almost strongly minimal, which was the  first natural occurring example of this kind of theories. Furthermore, as we will show in this paper (Lemma \ref{L:Equivalence}), it turns out that Podewski and Reineke's result implies that the ring $\C[x]/(x^2)$ is also an example of an uncountably categorical theory without the Canonical Base Property (CBP). 
		
	The CBP is a model-theoretic property introduced formally by Moosa and Pillay \cite{MP08}. Its formulation was motivated by the interpretation of Pillay \cite{aP02} of a result of Campana and Fujiki in compact complex spaces, and similar results of Pillay and Ziegler \cite{PZ03} for enriched fields in characteristic $0$. In particular, they proved the CBP for the finite rank part of the theory of differentially closed fields of characteristic $0$, which allowed them to simplify Hrushovski's proof of the Mordell-Lang conjecture for function fields in characteristic $0$.  We refer to \cite{zC12,zC14} for other applications and for an exposition of results around the CBP in the literature. 
		
	While it was conjectured that finite Morley rank theories satisfy the CBP, in \cite{HPP13} Hrushovski, the second author of this paper and Pillay presented an uncountably categorical structure without the CBP. The structure from \cite{HPP13} is an additive cover of the field of complex numbers which is interpretable in $\mathrm{ACF}_0$. This approach was also used in \cite{mL22} to produce other examples without the CBP. Also, using additive covers but of Baudisch's Lie algebras \cite{aB09}, Blossier and Jimenez \cite{BJ22} produced another example which in addition is CM-trivial. So, it does not interpret an infinite field, while the previous examples do.
	
	Therefore, at the moment of writing there are already several uncountably categorical structures without the CBP. Nonetheless, while these are produced from algebraic structures, none of these is an algebraic object per se, except the one in \cite{HPP13} after the reinterpretation given in \cite[Section 3]{mL23}. In fact, bearing this reinterpretation in mind, the example from \cite{HPP13} is essentially the ring $\C[x]/(x^2)$.
	
	The purpose of this paper is to give a large variety of natural algebraic examples whose theories have finite Morley rank  but do not satisfy the CBP. These examples will be rings and groups. 
	
	In Section \ref{s:Triangular} we introduce a very general framework in which all examples will fit. The construction is elementary: given two non-associate rings $R$ and $M$ together with two group homomorphisms $R\to \mathrm{End}_\Z(M)$ for left and right multiplication, we introduce the ring $\Lambda(R,M)$ whose domain is $R\times M$ with addition coordinatewise and with multiplication given by
	\[
	(r,m) \cdot (r',m') = (rr' , r\cdot m' + m\cdot r' + mm').
	\] 
	For instance, when $M$ has trivial multiplication the ring $\Lambda(R,M)$ can be naturally seen as a subring of matrices with the usual operations. This construction is then applied in Section \ref{s:NonA} to obtain our first examples of finite Morley rank rings without the CBP. There we prove that the  toy example of a finite Morley rank theory without the CBP is the ring theory of $\Lambda(k,V)$, where $V$ is a finite-dimensional vector space over an algebraically closed field $k$ of characteristic $0$ (Corollary \ref{C:Vector}). More generally, the above construction applies as well to Lie algebras. For a Lie algebra $\g$, it turns out that $\Lambda(\g,\g^+)$ is the semidirect sum $\g\ltimes \g^+$, where $\g^+$ stands for the abelian group $\g$ with trivial Lie bracket. In this case we prove the following general statement:
	
	\begin{thm}[Theorem \ref{T:Lie}]\label{T:A}
		Let $\g$ be a finite-dimensional Lie $k$-algebra over an algebraically closed field $k$ of characteristic $0$. If $ \g$ is non-abelian, then the Lie ring $\g\ltimes \g^+$ has finite Morley rank and does not satisfy the CBP.
	\end{thm}
	
	A similar result is obtained replacing $\g$ by Baudisch's Lie $\F_q$-algebra \cite{aB09}, where $q$ is a $p$-power for $p>2$ (Theorem  \ref{T:Lie_B}). In particular, applying the corresponding result to a Heisenberg Lie $\C$-algebra or to a Baudisch Lie $\F_q$-algebra, we obtain the existence of $2$-step nilpotent Lie rings of finite Morley rank without the CBP. Using the Baker-Campbell-Hausdorff formula, we then show:
	\begin{thm}[Corollary \ref{C:CBP-Group}]
		There is a $2$-step nilpotent (CM-trivial) connected group of finite Morley rank whose theory in the pure language of groups does not satisfy the CBP. 
	\end{thm}

	Finally, in Section \ref{s:Rings} we study which commutative unitary rings of finite Morley rank satisfy the CBP. Using results of Cherlin and Reineke \cite{CR76} and the construction from Section \ref{s:Triangular}, we obtain a purely algebraic classification of this under the assumption that the ring has no finite-index proper ideals. First we observe that it suffices to analyse local rings (see Lemma \ref{L:Equivalence}) and that in this case the CBP is equivalent to being almost strongly minimal. In that context we generalise the result of Podewski and Reineke \cite{PR74} and prove the following:

	\begin{thm}[Theorem \ref{T:Charac0}]
	Let $R$ be a commutative unitary local ring of finite Morley rank and suppose that it has characteristic $0$. If $R$ is not a field, then it is not almost strongly minimal nor satisfies the CBP.
	\end{thm}

	On the contrary, every commutative unitary local ring of positive characteristic is almost strongly minimal whenever its residue field is infinite (Theorem \ref{T:Charac+}), reinforcing the intuition behind the question from \cite[p. 874]{HPP13}. Altogether, we prove:
\begin{thm}[Corollary \ref{C:Charac}]
	Let $R$ be a commutative unitary ring of finite Morley rank without finite-index proper ideals. The theory of $R$ satisfies the CBP if and only if $R$ is a direct product $R_1\times \ldots \times R_n$ where each $R_i$ is a field or a local ring of positive characteristic. 
\end{thm}
	
	The situation is drastically distinct for commutative unitary local rings of finite Morley rank with a finite residue field, which are necessarily of positive characteristic. In this case, there is no algebraic description available nor a classification from the point of view of geometric stability. Nonetheless, we give a characterization for uncountably categorical equicharacteristic local rings from this family that are almost strongly minimal (Proposition \ref{P:AlmostStronglyMinimal}) and for those that satisfy the CBP (Theorem \ref{T:CBP}). On the other hand, we conclude the paper by giving a CM-trivial example of a local ring without the CBP, using Theorem \ref{T:Lie_B}.

\begin{thm}[Theorem \ref{T:CM-trivialRing}]
	There is a commutative unitary local ring with a finite-index residue field whose theory is CM-trivial of finite Morley rank but does not satisfy the CBP.
\end{thm}

	Basic notions and facts concerning the CBP are recalled in Section \ref{s:CBP}. Nonetheless, throughout the paper, we assume some familiarity with finite Morley rank (groups) and stability theory. We refer the reader to \cite{TZ12} for an excellent introduction.
	
	\subsection*{Acknowledgements} Both authors are in debt to Amador Martin-Pizarro for encouraging them to initiate this project.
	
	\section{The Canonical Base Property}\label{s:CBP}
	
	In this section we recall the definition of the canonical base property as well as some known, but useful, facts. 	
	In view of our interests, we shall restrict our attention to stable theories of finite Morley rank. So, we will tailor the definition and facts to finite Morley rank structures. We refer the reader to \cite[Section 3.0]{fW97} for a more general presentation and for further details.
	
	Fix a finite Morley rank saturated structure and work with imaginaries, so definable means indeed interpretable. Let $\Pp_0$ be a family of definable sets and denote by $\Pp$ the family of non modular strongly minimal definable sets.
		
	\begin{definition}
	A definable set $X$ is {\em $\Pp_0$-internal} if there are finitely many sets $X_1,\ldots,X_k\in \Pp_0$ and a definable surjection $f: \bigcup_{i=1}^k X_i^{n_i}\to X$. 
	\end{definition}
	
	The above definition is equivalent to saying that $X\subset \dcl(A,X_1,\ldots,X_n)$ for some set of parameters $A$ containing the parameters of $X_1,\ldots,X_n\in \Pp_0$. Also, this definition extends to partial types as follows: a partial type $\pi(x)$ is $\Pp_0$-internal if it implies a definable set which is $\Pp_0$-internal. This agrees with the usual definition of (almost) internality: 
	
	\begin{definition}
	A partial type $\pi(x)$ over $A$ is (almost) $\Pp_0$-internal if there is a set $B\supset A$ such that any realisation of $\pi$ is definable (algebraic) over $B$ and finitely many realisations of definable sets over $B$ from $\Pp_0$. We also say that $\pi(x)$ over $A$ is {\em $\Pp_0$-analysable}  if for every realisation $a$ of $\pi$ there are $a_1,\ldots,a_n\in \dcl(A,a)$ such that $\tp(a_{i+1}/A,a_1,\ldots,a_i)$ is $\Pp_0$-internal for $1\le i<n$ and $a\in \dcl(A,a_1,\ldots,a_n)$. 
	\end{definition}
Note that almost internality is a generalisation of almost strong minimality. Indeed, a partial type is almost strongly minimal if it is almost internal to a strongly minimal set. 

In our context, every type is analysable to the family of strongly minimal sets. 
%
The following fact is probably folklore, but we could not find a reference. So, we include a proof for the sake of completeness. 
	\begin{fact}\label{F:Internal}
		Let $\Pp_0$ be a family of definable sets and let $\Pp_1$ be a family of strongly minimal sets. Assume that a definable set $X$ is $\Pp_0$-analysable and almost $\Pp_1$-internal. Then it is almost $\Pp_0$-internal.
	\end{fact}
\begin{proof}
Let $X$ be definable over $A$ and let $a\in X$ be arbitrary. Since $\tp(a/A)$ is almost $\Pp_1$-internal, there exists some superset $B\supset A$ which is independent from $a$ over $A$ and some finite tuple $\bar c= (c_1,\ldots,c_k)$ such that $a\in \acl(B,\bar c)$ and each type $\tp(c_i/B)$ is strongly minimal. By refining the choice of $\bar c$ and $B$ if necessary, we may assume that $\bar c$ is an independent tuple over $B$ and that $\acl(B,a) = \acl(B,\bar c)$. It then follows that each type $\tp(c_i/B)$ is  $\Pp_0$-analysable, and so almost $\Pp_0$-internal since they have rank $1$. Hence, the type $\tp(a/A)$ is almost $\Pp_0$-internal, as desired.
\end{proof}

Concerning definable groups, the fundamental result in this setting is the following theorem due to Hrushovski. 

\begin{fact}
	Let $G$ be a definable group of finite Morley rank. Then $G$ is $\Pp_0$-analysable if and only if there is a descending sequence \[ G=G_0 \unrhd G_1 \unrhd \dots \unrhd G_k =\{1_G\}
	\]
	of definable normal subgroups such that each quotient $G_i/G_{i+1}$ is $\Pp_0$-internal.
\end{fact}

We now recall the definition of the canonical base property: 
	\begin{definition}
		A theory of finite Morley rank has the {\em canonical base property} (CBP, in short) if whenever $a$ and $b$ are finite tuples in a saturated model we have that $\tp(\cb(\stp(a/b))/a)$ is almost $\Pp$-internal.
	\end{definition}

Chatzidakis proved in \cite[Proposition 1.10]{zC12} (\cf \cite[Corollary 5.2]{PW13}) that in fact the type $\tp(\cb(\stp(a/b))/a)$ is always $\Pp$-analysable. Using this, Kowalski and Pillay \cite[Section 4]{KP06} proved that in general, if $G$ is a definable connected group of finite Morley rank, then $G/Z(G)$ is $\Pp$-analysable. Furthermore, under the assumption of the CBP, they observed that $G/Z(G)$ is almost $\Pp$-internal. An inspection of their proof yields the following useful result for non-associative rings. By a (possibly) {\em non-associative ring} we mean an algebraic structure which satisfies all axioms of rings except possibly the associativity property for multiplication. That is, it is an abelian group together with a multiplication operation that distributes over addition. In this context, given a subset $A$ of a non-associative ring $R$, we denote by $\ann_R(A)$ the two-sided annihilator of $A$ in $R$, that is,
\[
\ann_R(A) = \left\{ r \in R \ | \ ra= 0 =ar \text{ for every $a\in A$} \right\}.
\]
The annihilator of $R$ is then $\ann(R):=\ann_R(R)$.

	\begin{prop}\label{P:Ann}
	Let $R$ be a non-associative ring of finite Morley rank definable in a theory with the CBP.  Suppose that $R$ is connected as an additive group. The definable additive group $R/\ann(R)$ is almost $\Pp$-internal.
	\end{prop}
	\begin{proof}
	Fix an arbitrary generic element $a\in R$ and consider the definable maps $x\mapsto ax$ and $x\to xa$, which are  group homo\-morphisms by the distributive law. Set 
	\[
	H_a=\left\{ (x,ax,xa) \ | \ x\in R \right\}
	\]
	and note that it is a subgroup of $R\times R\times R$  which is in definable bijection with $R$. So, it is connected. Hence, its canonical parameter is almost $\Pp$-internal by \cite[Lemma 2.6]{aP02}. On the other hand, distributivity yields that the canonical parameter of $H_a$ is interdefinable with $a+\ann(R)$. So, as $a+\ann(R)$ is a generic element of $R/\ann(R)$, it follows that  $R/\ann(R)$ is almost $\Pp$-internal.	
	\end{proof}
	 
	 \section{The basic construction: triangular rings}\label{s:Triangular}
	 In this section we present a very general construction in ring theory which will encompass all examples appearing in the paper. For this purpose, it will be convenient to work with the following concept of module over an abelian group, see \cite[Section A.1.1]{BN94}.
	 \begin{definition}
	 	Let $R$ be an abelian group. We say that $M$ is a {\em left $R$-module} if it is an additive group $M$ together with an operation $R\times M\to M$ that distributes over addition of $M$, {\it i.e.} for all $r,r'\in R$ and $m,m'\in M$ the following holds:
	 	\begin{enumerate}[(i)]
	 		\item $(r+r')\cdot m = r\cdot m + r'\cdot m $ and  \item $r\cdot (m+m') = r\cdot m + r\cdot m'$.
	 	\end{enumerate}
	 	Likewise, we define the concept of {\em right $R$-module} with right multiplication. 
	 \end{definition} 
	 
	 Notice that a (left) $R$-module is nothing else than an abelian group $M$ together with a group homomorphism $R \to \mathrm{End}_\Z(M)$ of abelian groups, or equivalently with a (left) bilinear map  $R\times M\to M$. Along the paper, we will use this notion for non-associative rings $R$ and $M$. 
	 
	 Next, we introduce the main ring theoretic construction, which is a generalisation of triangular rings as presented in \cite[Example 1.14]{tL01}. We adapt this construction to our more general setting.
	 \begin{definition}
	 	Let $R$ and $M$ be two non-associative rings and suppose that $M$ is a left and right $R$-module. The {\em (general) triangular ring} associated to $R$ and $M$, denoted by $\Lambda(R,M)$, is the set $R\times M$ equipped with addition coordinatewise and multiplication
	 	\[
	 	(r_1,m_1) \cdot (r_2,m_2) = (r_1r_2 , r_1\cdot m_2 + m_1\cdot r_2 + m_1m_2).
	 	\]
	 \end{definition} 
	 It is routine to verify that $\Lambda(R,M)$ with these operations becomes a non-associative ring. As pointed out above, the terminology comes from the following example in ring theory.
	 
	 \begin{example}
	 	Let $R$ and $M$ be two non-associative rings and suppose that $M$ is a left and right $R$-module. Assume further that the ring $M$ has trivial multiplication, {\it i.e.} it is simply a left and right $R$-module. In that case, the triangular ring $\Lambda(R,M)$ corresponds to the non-associative ring
	 	\[
	  \left\{ 
	 	\begin{pmatrix} 
	 	r & m \\ 0 & r 
	 	\end{pmatrix}
	 	\ \Big| \ r\in R \text{ and } m\in M
	 	\right\}
	 	\]
	 	with the usual matrix operations. 
	 \end{example}
	 
	 The following result captures one of the main relevant features (for our purposes) of triangular rings: certain additive maps from $R$ to the annihilator of $M$ lift to automorphisms of the triangular ring. This will play an essential role throughout the paper. 
	 
	 \begin{lemma}\label{L:Aut}
	 	Let $\Lambda$ be the triangular ring corresponding to two non-associative rings $R$ and $M$.  Assume that there exists an  additive map $\delta:R\to \ann(M)$ satisfying for every $r,r' \in R$ that 
	 	\[
	 	\delta (r r') = r\cdot \delta(r') + \delta(r)\cdot r'.
	 	\] Then, the map
	 	\[
	 	\sigma: \Lambda \to \Lambda, \ a=(r,m) \mapsto  a+(0,\delta(r)) 
	 	\]
	 	is an automorphism of the ring $\Lambda$ which fixes pointwise $A_0\times M$ whenever $\delta$ vanishes on $A_0\subset R$.  Furthermore, we also have that $\sigma(1_\Lambda)=1_\Lambda$ whenever $R$ is unital.
	 \end{lemma}
	 \begin{proof}
	 	It is easy to check that $\sigma$ is an additive bijection. We prove that it is multiplicative. Note first that for $a=(r,m)$ and $a'=(r',m')$ in $\Lambda$ we have 
	 	\[
	 	m\cdot \delta(r') + \delta(r)\cdot m' + \delta(r) \delta(r') = 0,
	 	\]
	 	since $\delta(R)\subset \ann(M)$. So, we have
	 	\begin{align*}
	 	\sigma(a)\cdot \sigma(a') & = (a+(0,\delta(r))\cdot (a'+(0,\delta(r')) \\& = a\cdot a' + (0,r\cdot \delta(r') + \delta(r)\cdot r')  \\  
	 	& = a\cdot a' + (0,\delta(r \cdot r')),
	 	\end{align*}
	 	where the third equality holds by assumption.	
	 \end{proof}
	 
	 \begin{remark}
	 	In general, in the statement above, the non-associative ring $R$ cannot have trivial multiplication for a meaningful additive map $\delta:R\to \ann(M)$ to exist. On the contrary,  {\it a priori} there is no restriction on $M$. In fact, if $M$ has trivial multiplication, then $\ann(M) =M$ and consequently $\delta$ takes values on $M$.
	 \end{remark}

	 The other main ingredient is that the triangular ring corresponding to $R$ and $M$ has a suitable analysis. Before stating the result, we define the two-sided {\em annihilator} of $M$ in $R$ as
	 \[
	 \ann_R(M) = \left\{ r\in R \ | \ r\cdot M =\{0\} = M\cdot r \right\}.
	 \]
	 Likewise, define the two-sided annihilator of $R$ in $M$ as
	 \[
	 \ann_M(R) = \left\{ m\in M \ | \ R\cdot m = \{0\} = m\cdot R \right\}.
	 \]
	 Both are clearly subgroups of $R$ and $M$ respectively. Also, notice that for $M=R$ the above is the annihilator of $R$. 
	 \begin{lemma}\label{L:Analysis}
	 	Let $R$ and $M$ be two non-associative rings and let $\Lambda$ be the triangular ring corresponding to them. Suppose that $\Lambda$ has finite Morley rank and assume further that $\ann(R) = \ann_R(M) = \ann_{R}(\ann(M))$. Then:
	 	\begin{enumerate}
	 		\item The annihilator of $\Lambda$ is $\ann(R)\times \big(\ann(M)\cap \ann_M(R)\big)$.
	 		\item The group $\Lambda/\ann(\Lambda)$ is analysable to a definable subgroup of $\{0\}\times  M$. 
	 		\item If the ring theory of $\Lambda$ satisfies the CBP, then $\Lambda/\ann(\Lambda^\circ)$ is almost internal to a definable subgroup of $\{0\}\times M$.
	 	\end{enumerate}
	 	
	 \end{lemma}
	 \begin{proof} {\em (1)} The inclusion $\supset$ is straightforward using that $\ann(R)  \subset \ann_R(M)$. For the other, consider an arbitrary element $(r,m)\in \ann(\Lambda)$. By considering the product of $(r,m)$ and any $(r',0)\in\Lambda$ we immediately see that $r\in \ann(R)$. So, by assumption we get that $r\in \ann_R(M)$ as well. It then follows for any $m'\in M$ that 
	 	\[
	 	(r,m)\cdot (0,m') =  (0,mm') \ \text{ and  } \ (0,m')\cdot (r,m) =  (0,m'm).
	 	\]
	 	Hence we deduce that $m\in \ann(M)$. Also, for any $r'\in R$ we see that 
	 	\[
	 	(r,m)\cdot (r',0) =  (0,m\cdot r') \ \text{ and  } \ (r',0)\cdot (r,m) =  (0,r'\cdot m).
	 	\]
	 	So,  we get $m\in \ann_M(R)$, as desired.
	 	
	 	\noindent {\em (2)} To ease notation, given an element $a\in \Lambda$ denote by $\ell_a:\Lambda\to \Lambda$ and $\rho_a:\Lambda\to \Lambda$ the additive homomorphisms defined by $\ell_a(x)= a\cdot x$ and $\rho_a(x) = x\cdot a$. 
	 	
	 	Let $m_1,\ldots,m_k\in \ann(M)$ be such that 
	 	$
	 	\ann_\Lambda(\{0\}\times \ann(M)) = \bigcap\nolimits_j \ann_{\Lambda}((0,m_j)).
	 	$ 
	 	For each $j$, set $b_j = (0,m_j)$ and consider the additive homomorphism 
	 	\[
	 	f_1 : \Lambda \to \prod\nolimits_{j=1}^{2k} \Lambda , \ x \mapsto \left( \ell_{b_1}(x),\rho_{b_1}(x),\ldots, \ell_{b_k}(x),\rho_{b_k}(x)\right).
	 	\] 
	 	We clearly have that $\mathrm{im}(f_1) \subset \prod_{j} \{0\}\times M$. Furthermore, the kernel of $f_1$ is  
	 	\[
	 	\Lambda_1:=\ann_\Lambda(\{0\}\times \ann(M)) = \ann_R(\ann(M)) \times M =\ann(R)\times M.
	 	\] 
	 	It follows that $\Lambda/\Lambda_1$ is definably isomorphic to the additive subgroup $\mathrm{im}(f_1) $ and hence it is internal to the definable subgroup $\sum_j \mathrm{im}(\ell_{b_j}) + \sum_j \mathrm{im}(\rho_{b_j})$ of  $\{0\}\times M$.
	 	
	 	Now, we prove that $\Lambda_1/\ann(\Lambda)$ is internal to a definable subgroup of  $\{0\}\times M$. Let $a_1,\ldots,a_s\in \Lambda$ be such that $\ann (\Lambda) = \bigcap_j \ann_{\Lambda}(a_j)$. We know by (1) that 
	 	\[
	 	\ann (\Lambda) = \ann(R)\times \big(\ann(M)\cap \ann_M(R)\big)\subset \Lambda_1.
	 	\]
	 	Similar as above, consider the additive homomorphism 
	 	\[
	 	f_2 : \Lambda_1  \to \prod\nolimits_{j=1}^{2s} \Lambda , \ x \mapsto \big( \ell_{a_1}(x),\rho_{a_1}(x),\ldots, \ell_{a_s}(x),\rho_{a_s}(x)\big).
	 	\] 
	 	Note that $\mathrm{im}(f_2)\subset\prod_{j} \{0\}\times M$ since $\Lambda_1\subset \ann(R)\times M$. Also, the kernel of $f_2$ is $\ann(\Lambda)$. So, we get that $\Lambda_1/\ann(\Lambda)$ is definably isomorphic to $\mathrm{im}(f_2)$, which implies that it is internal to the definable subgroup $\sum_j \mathrm{im}(\ell_{a_j}) + \sum_j \mathrm{im}(\rho_{a_j})$ of $\{0\}\times M$. Consequently, the quotient $\Lambda/\ann(\Lambda)$ is analysable in two steps to a definable subgroup of $\{0\}\times M$.
	 	
	 	\noindent {\em (3)} It follows from  Proposition \ref{P:Ann} that $\Lambda^\circ/\ann_{\Lambda^\circ}(\Lambda^\circ)$ and hence $\Lambda/\ann(\Lambda^\circ)$ are almost $\Pp$-internal. By (2) we have that $\Lambda/\ann(\Lambda)$ is analysable to a definable subgroup of $\{0\} \times M$, say $H$. So, we deduce that $\Lambda/\ann(\Lambda^\circ)$ is analysable to $H$ as well, since $\ann(\Lambda)\subset \ann(\Lambda^\circ)$ and so $\Lambda/\ann(\Lambda)$ projects definably onto $\Lambda/\ann(\Lambda^\circ)$. Altogether, we obtain that $\Lambda/\ann(\Lambda^\circ)$ is almost internal to $H$ by Fact \ref{F:Internal}.  \end{proof}
	 
	 \begin{remark}
	 The proof above yields that $\Lambda/\ann(\Lambda)$ is analysable to a definable subgroup of $\{0\} \times \big( R\cdot M \cup M\cdot R\cup M\cdot M)$. Likewise, we obtain a similar statement in (3).
	 \end{remark}

	\section{Triangular rings without the CBP}\label{s:NonA}
	 
		In this section we obtain several non-associative rings whose theories do not satisfy the CBP. We first study (general) triangular rings obtained from finite-dimensional algebras and later from Lie algebras.  As a consequence, we see that the triangular ring $\Lambda(k,V)$ is a toy example of a finite Morley rank theory without the CBP, where $V$ is a finite-dimensional vector space over an algebraically closed field $k$ of characteristic $0$.
		
	\subsection{Non-associative algebras} Recall that a non-associative algebra $A$ is a vector space over a field $k$ together with a distributive multiplication. Thus, we can regard $A$ as a non-associative ring which in additon is a left and right $k$-module, where left and right multiplication by scalars agree. Hence, we can consider its general triangular ring $\Lambda(k,A)$.
	
	\begin{remark}\label{R:Interpret}
	Let $A$ be a finite-dimensional algebra over an algebraically closed field $k$. Then $A$ is definable in the field $k$. Indeed, the additive group structure of $A$ is clearly definable on $k$. Moreover, note that  multiplication is determined by the equations \[
		v_i \cdot v_j  = \sum_{\ell=1}^n c_{i,j,\ell} v_\ell,
		\] 
		where $v_1,\ldots,v_n$ is a basis of $A$ over $k$ and $c_{i,j,\ell}\in k$. So, the algebra multiplication is also definable on $k$. Therefore, the ring theories of both non-associative rings $A$ and $\Lambda(k,A)$ have finite Morley rank. 
	\end{remark}

	\begin{theorem}\label{T:Algebra}
	Let $A$ be an algebra over an algebraically closed field $k$ of characteristic $0$ and assume further that $\ann(A)$ is non-trivial. If $\Lambda(k,A)$ has finite Morley rank, then its theory does not satisfy the CBP. In particular, the theory of $\Lambda(k,A)$ does not satisfy the CBP whenever $A$ is finite-dimensional over $k$.
	\end{theorem}
	\begin{proof} We first observe that we may assume that $\Lambda:=\Lambda(k,A)$ is saturated and uncountable, hence $k$ is uncountable as well. 
	
	Now, suppose aiming for a contradiction that $\Lambda$ satisfies the CBP. Since $k$ acts faithfully on $A$, both $k$ and $A$ satisfy the hypothesis of Lemma \ref{L:Analysis} and therefore our assumption yields that $\Lambda$ is almost internal to a definable subgroup of $\{0\}\times A$. So, there exists some $a_1=(\alpha_1,v_1),\ldots,a_n=(\alpha_n,v_k)\in \Lambda$ witnessing this.
	
	As $k$ is uncountable and algebraically closed of characteristic $0$, we can find some non-trivial derivation $\delta_0:k\to k$ that vanishes on $\mathbb Q(\alpha_1,\ldots,\alpha_n)^{\rm alg}\subset k$, see for instance \cite[Ch. 2, Section 17]{ZS58}. Fix some $x_0\in \ann(A)\setminus \{0\}$ and define the map 
	\[
	\delta: k\to A, \ \alpha  \mapsto \delta_0(\alpha) \cdot x_0. 
	\]
	It is straightforward to verify that $\delta$ is an additive map satisfying that
	\[
	\delta (\alpha\beta) = \delta(\alpha)\cdot \beta + \alpha\cdot \delta(\beta)
	\] 
	So, by Lemma \ref{L:Aut} we obtain an automorphism $\sigma\in \Aut(\Lambda)$ fixing pointwise $\{0\} \times A$ and $\
	a_1,\ldots,a_n$. On the other hand, for any element $\alpha\in k\setminus  \mathbb Q(\alpha_1,\ldots,\alpha_n)^{\rm alg}$ such that $\delta_0(\alpha)\neq 0$, we have that the $(i+1)$-iterated image of $a=(\alpha,0)$ under $\sigma$ is
	\[
	\sigma^{i+1}(a) =  \sigma((\alpha,i\cdot \delta(\alpha))) = (\alpha, i\cdot \delta(\alpha) + \delta(\alpha)) = (\alpha,(i+1)\cdot \delta(\alpha)).
	\]
	In particular, it follows that $\sigma^i(a) \neq \sigma^j(a)$ for any $0<i<j$. Otherwise, we would get $(j-i)\cdot \delta_0(\alpha)\cdot x_0=0$ which would imply that $x_0=0$. Hence, we get that $a \not\in \acl(a_1,\ldots,a_n,\{0\}\times A)$ and therefore we obtain the desired contradiction.
	\end{proof}
	
	In view of the above result, we obtain a toy example of a finite Morey rank structure without the CBP. Regarding a vector space $V$ over a field $k$ as a $k$-algebra with trivial product we readily have the following:
	
	\begin{cor}\label{C:Vector}
	Let $V$ be a finite-dimensional vector space over an algebraically closed field $k$ of characteristic $0$. The theory of $\Lambda(k,V)$ does not satisfy the CBP.
	\end{cor}
	In fact, this yields that the ring $\C[x]/(x^2)$ does not satisfy the CBP, since it is cleary isomorphic to $\Lambda(\C,\C^+)$. In Section \ref{s:Rings} we will see that the argument above goes through for arbitrary commutative unitary local rings of characteristic $0$. 
	
	\subsection{Lie algebras} \label{s:Lie} Recall that a {\em Lie ring} $L$ is a non-associative ring in which multiplication $[ \cdot \,, \cdot]$, called Lie bracket, is alternative and satisfies the Jacobi identity, {\it i.e.} $[x,x]=0$ for every $x\in L$ and
		\[
		[x,[y,z]] + [y,[z,x]] + [z,[x,y]] = 0
		\] 
		for every $x,y,z\in L$.  A {\em Lie algebra} $\g$ over a field $K$ is a Lie ring which is a $K$-vector space and where the bracket operation  $[\cdot \,, \cdot ]$  is bilinear also with respect to $K$. It follows from bilinearity and alternativity that $[x,y]=-[y,x]$ for every $x,y\in L$.
	
	Every Lie algebra is an example of a Lie ring, where the linear structure is omitted. Also,  every associative ring (algebra) can be made into a Lie ring (algebra) by setting $[x,y]=xy-yx$. In particular, given a vector space $V$, its set of endo\-morphisms $\mathrm{End}(V)$ is an associative algebra. The corresponding Lie algebra is the {\em general linear Lie algebra} which is denoted by $\mathfrak{gl}(V)$. As usual, we denote by $\mathfrak{gl}_n(K)$ the general linear Lie algebra when $V$ is an $n$-dimensional vector space over $K$. 
	
	A {\em linear Lie ring} is a subring of $\mathfrak{gl}(V)$ with the commutator as the Lie bracket. In fact, by  Ado's Theorem every finite-dimensional Lie algebra over an algebraically closed field of characteristic $0$ is a linear Lie ring. The same is true in positive characteristic which is due to Iwasawa, see \cite[Chapter VI]{nJ62} for proofs.


	\begin{remark}\label{R:Lie-CBP}
		There are many natural examples of Lie rings of finite Morley rank that satisfy the CBP for trivial reasons.  
		\begin{itemize}
			\item Every abelian Lie ring of finite Morley is an abelian group together with a trivial Lie bracket, so its theory in the language of Lie rings is one-based and hence it satisfies the CBP.
			\item Every simple non-abelian Lie ring of finite Morley, {\it i.e.} a Lie ring without proper ideals, is almost strongly minimal, see the proof of \cite[Lemma 3]{fW01}. Hence, its theory in the language of Lie rings also satisfies the CBP. 
		\end{itemize}
	\end{remark}
	
	\begin{prop}\label{P:2-dim}
		Every $2$-dimensional Lie algebra $\g$ over an algebraically closed field $K$ is almost strongly minimal regarded as Lie ring. 
	\end{prop}
	\begin{proof}
		Let $\g$ be a $2$-dimensional Lie $K$-algebra, which we may assume to be non-abelian. This Lie algebra has a basis $x,y$ such that its  Lie bracket is $[x,y]=x$, see \cite[p. 11]{nJ62}. Thus, the derived Lie subalgebra $[\g,\g]=[x,\g]$ is strongly minimal. Given a ring automorphism $\sigma\in \Aut(\g)$ fixing pointwise $[\g,\g]$ and the basis $x,y$, we see for any $a,b\in K$ that 
		\[
		\sigma(a)x = [-y,\sigma(a)x+\sigma(b)y)]= \sigma([-y,ax+by] ) = [-y,ax+by] = ax,
		\]
		so $\sigma(ax)=ax$. Likewise, one verifies that $\sigma(by)=by$, hence $\sigma(ax+by) =ax+by$. This shows that $\g\subset \dcl([\g,\g],x,y)$ and therefore $\g$ is almost strongly minimal.
	\end{proof}
	
	To obtain Lie rings without the CBP we will consider triangular rings $\Lambda(\g,\g^+)$ with respect to a Lie  algebra $\g$, where $\g^{+}$ denotes the additive group of $\g$. 
	It turns out that $\Lambda(\g,\g^+)$ is the semidirect sum of a Lie algebra $\g$ with respect to the adjoint representation $\ad:\g\to \mathfrak{gl}(\g^+)$, where $\ad(a)(x) = [a,x]_\g$.
	\begin{definition}
	Let $\g$ be a Lie algebra. The {\em semidirect sum} $\g\ltimes \g^{+}$ with respect to the adjoint representation  is the Lie algebra $\g\ltimes \g^+$ whose additive group is $\g\times \g$ and its  Lie bracket is 
	\[
	\big[(a,x),(b,y) \big] = \big( [ a,b ]_\g , \ad(a)(y) - \ad(b)(x)  \big) = \big( [ a,b ]_\g , [a,y]_\g + [x,b ]_\g   \big).
	\]
	\end{definition} 
	Note that $\g\ltimes \g^+ = \Lambda(\g,\g^+)$. Also,  since the kernel of $\ad$ is the center $\z(\g)$ of $\g$, we have that $\z(\g) = \ann_\g(\g^+)$ and  $\z(\g\ltimes \g^+) = \z(\g)\times \z(\g)$  by Lemma \ref{L:Analysis}.
	
	\begin{remark} This construction can be generalised to arbitrary representations as follows. To do so, one needs to consider the Lie algebra of derivations $\mathrm{Der}(\g)$ of a Lie algebra, where the Lie bracket is the commutator. Now, given two Lie algebras $\g_1$ and $\g_2$ and an homomorphism $\rho:\g_1\to \mathrm{Der}(\g_2)$, one defines the {\em semidirect sum} $\g_1\ltimes_\rho \g_2$ of $\g_1$ and $\g_2$ with respect to $\rho$ to be the additive group $\g_1\times \g_2$ together with the Lie bracket 
		\[
		\big[ (a,x) , (b,y) \big]  = \big( [ a,b ]_1 ,  [x,y]_2 + \rho(a)(y) - \rho(b)(x)  \big).
		\]
	This is precisely $\Lambda(\g_1,\g_2)$. However, not every semidirect sum of Lie algebras fails to have the CBP. Indeed, by  Proposition \ref{P:2-dim} the $2$-dimensional Lie $\C$-algebra  $\mathfrak{aff}(\C) = \C\ltimes_{\rm id} \C$ is almost strongly minimal.
	\end{remark}

	\subsubsection{Finite-dimensional Lie algebras} 
	Let $\g$ be a linear Lie algebra over a field $K$ and let $e_{i,j}$ denote the standard matrix unit whose sole non-zero entry $1$ is in the $(i,j)$ position. Every element $a\in \g$ can be written as $a=\sum_{i,j} a_{i,j}\cdot e_{i,j}$ for some coefficients $a_{i,j}\in K$. 
	Now, given a map $f:K\to K$, define 
	\[
	\hat{f}:\g\to \mathfrak{gl}_n(K), \ a=\sum\nolimits_{i,j}a_{i,j} \cdot e_{i,j} \mapsto \hat{f}(a) := \sum\nolimits_{i,j} f(a_{i,j})\cdot e_{i,j}, 
	\]
	that is, the map $\hat f$ consists of applying $f$ coordinatewise. 
	\begin{lemma}\label{L:Leibniz}
		Let $\g\subset \mathfrak{gl}_n(K)$ be a linear Lie ring over a field $K$. If $\delta:K\to K$ is a derivation on $K$, then $\hat\delta:\g\to \mathfrak{gl}_n(K)$ is an additive map satisfying the Leibniz rule. Furthermore, if  $v_1,\ldots,v_r$ is a basis of $\g$ over $K$ and $\delta$ vanishes on the coefficients of each $v_k$ with respect to the basis $e_{i,j}$ of $\mathfrak{gl}_n(K)$, then $\mathrm{im} (\hat \delta) \subset \g$. 
	\end{lemma}
	\begin{proof}
		Clearly, the map $\hat\delta$ is additive, as so is $\delta$. To verify that it satisfies the Leibniz rule, let $a=\sum_{i,j} a_{i,j}\cdot e_{i,j}$ and $b=\sum_{i,j} b_{i,j}\cdot e_{i,j}$ in $\g$ and set $a\cdot b=c$, where $c=\sum_{i,j} c_{i,j}\cdot e_{i,j}$. Note that
		\[
		\delta(c_{i,j}) = \delta \left( \sum\nolimits_k a_{i,k} b_{k,j} \right)= \sum\nolimits_k \delta \left( a_{i,k} b_{k,j}  \right) =\sum\nolimits_k a_{i,k} \delta (b_{k,j}) + \delta(a_{i,k})b_{k,j}. 
		\]
		It then follows that $\hat\delta(a\cdot b) = a\cdot \hat\delta(b) + \hat\delta(a)\cdot b$. So, we have that
		\[
		\hat\delta([a,b]) =  \hat\delta(a\cdot b) - \hat\delta(b\cdot a) = a\cdot \hat\delta(b) + \hat\delta(a)\cdot b - (b\cdot \hat\delta(a) + \hat\delta(b)\cdot a) = [a,\hat\delta(b)] + [\hat\delta(a), b],
		\]
		as desired.
		
		Finally, suppose that $v_1,\ldots,v_r$ is a basis of $\g$ over $K$ and that $\delta$ vanishes on their coefficients with respect to the basis $e_{i,j}$ of $\mathfrak{gl}_n(K)$. Write each $v_k$ as $\sum_{i,j} c_{i,j,k} \cdot e_{i,j}$ with $c_{i,j,k}\in K$. Note that, since $\delta(c_{i,j,k})=0$ for every $i,j$ and $k$, for any $\lambda\in K$ and any $k$ we have
		\[
		\hat \delta(\lambda \cdot v_k) = \sum_{i,j} \delta(\lambda \, c_{i,j,k}) \cdot e_{i,j,k} = \sum_{i,j} \delta(\lambda) \, c_{i,j,k}\cdot e_{i,j} = \delta (\lambda) \cdot v_k,
		\] 
		and consequently $\mathrm{im}(\hat \delta) \subset \g$.
	\end{proof}
	
	Next, we state and prove our main result on linear Lie rings over an algebraically closed field of characteristic $0$.
	
	\begin{theorem}\label{T:Lie}
		Let $\g$ be a finite-dimensional Lie algebra over an algebraically closed field $K$ of characteristic $0$. If $\g$ is  non-abelian, then the Lie ring $\g\ltimes \g^+$ has finite Morley rank and its theory does not satisfy the CBP.
	\end{theorem}
	\begin{proof}
		By Ado's Theorem \cite[Chapter VI]{nJ62} we may assume that $\g$ is a linear Lie $K$-algebra. Also, as pointed out in Remarks \ref{R:Interpret}, the Lie algebra $\Lambda(\g,\g^+)=\g\ltimes \g^+$ is definable in $\mathrm{ACF}_0$. So, after replacing $K$ if necessary we may assume that $K$ is uncountable and hence $\g\ltimes \g^+$ is a saturated Lie ring of finite Morley rank. Furthermore, since the additive group of a vector space over $K$ is divisible, both $\g$ and $\g\ltimes \g^{+}$ are connected finite Morley rank Lie rings. 
		
		Suppose, to get a contradiction, that the Lie ring theory of $\g\ltimes \g^{+}$ satisfies the CBP. Therefore, we obtain by Lemma \ref{L:Analysis}  that $\g\ltimes \g^{+}/\z(\g\ltimes \g^+)$ is almost internal to a definable subgroup of $ \{0\}\times \g$. Let $A_0\subset \g\ltimes \g^{+}$ be a finite subset witnessing this, containing the parameters of a such subgroup. We show the existence of an element $(b,y)\in \g\ltimes \g^{+}$ such that the imaginary element $(b,y)+\z(\g\ltimes \g^{+})$ has an infinite orbit under the action of ring automorphisms of $\g\ltimes \g^{+}$ fixing pointwise $A_0$ and $\{0\}\times \g$.
		
		Let $(a_1,x_1),\ldots,(a_k,x_k)$ be the elements of $A_0$. As $\g$ is a linear Lie algebra over $K$, the elements $a_1,\ldots,a_k$ are matrices with coefficients  on $K$. Let $B\subset K$ be the finite subset consisting of these elements together with the coefficients of a basis of $\g$ with respect to the standard matrix units. By \cite[Theorem 4.2]{aN94}  we can find a non-trivial derivation $\delta$ on $K$ whose field of constants is precisely $\Q(B)^{\rm alg}$. Hence, we obtain by Lemma \ref{L:Leibniz} a non-trivial derivation  $\hat\delta$ on $\g$  that vanishes on $a_1,\ldots,a_k$. Furthermore, note that $\hat\delta$ only vanishes on a countable subset of $\g$, because $\Q(B)^{\rm alg}$ is countable.  
		
		\noindent \emph{Claim 1.}
		There exists some $b\in \g$ such that $\hat\delta(b) \not\in \z(\g)$.
		
		\noindent\emph{Proof of Claim 1.}
		Since $\g$ is not abelian, there is some $a\in \g$ such that $[a,\g]\neq \{0\}$.  In particular, the definable set $[a,\g]$ is uncountable and so $\hat\delta([a,\g])\neq \{0\}$, since $\hat\delta$ vanishes on a countable set. This yields that there must be some  $b\in \g$ such that $\hat\delta([a,b]) \neq 0$, so $\hat\delta(a)\not\in \z(\g)$ or $\hat\delta(b)\not\in \z(\g)$ by the Leibniz rule. 
		\hfill $\square${\tiny Claim 1}
		\medskip
		
		Now,  let $b\in \g$ be given by Claim 1 and consider the automorphism $\sigma$ of $\Lambda(\g,\g^+)=\g\ltimes \g^+$ associated to $\hat\delta$ given by Lemma \ref{L:Aut}. So, for an element $(c,z)\in \g\ltimes \g^+$ we have that $\sigma(c,z) = (c,z+\hat{\delta}(c))$. It then follows that $\sigma$ fixes pointwise $\{0\}\times \g$ and $A_0$, by the choice of $\delta$. On the other hand, we have for $n$ in $\N$ that
		\[
		\sigma^{n+1} (b,0) = \sigma (b,n\cdot \hat\delta(b)) =  (b,n\cdot \hat\delta(b) + \hat\delta(b)) = (b,(n+1)\cdot \hat\delta(b)).
		\]  
		We claim the following:

		\noindent \emph{Claim 2.}
		For $n<m$ we have that $\sigma^m(b,0) - \sigma^n(b,0)\notin \z(\g\ltimes \g^+)$.
		
		\noindent\emph{Proof of Claim 2.}
		It suffices to show that $(0,k\cdot \hat\delta(b))\not\in \z(\g\ltimes \g^+)$ for any $k\ge 1$. So, as $\z(\g\ltimes \g^+)= \z(\g)\times  \z(\g)$, we need to see that $k\cdot \hat\delta(b)   \not\in \z(\g)$. We already know that $\hat\delta(b)\not\in \z(\g)$ by the choice of $b$, so there is some $a\in \g$ such that $[\hat\delta(b),a]\neq 0$. Now, note that $\g$ is torsion-free, since it is a vector space over a field of characteristic $0$. Hence, we have that $[k\hat\delta(b),a] = k[\hat\delta(b),a] \neq 0$ for every $k\ge 1$, as desired. 
		\hfill $\square${\tiny Claim 2}
		\medskip
		
		The claim above implies that $\g\ltimes \g^+/\z(\g\ltimes \g^+)$ is not algebraic over $A_0\cup (\{0\}\times \g)$, which yields the desired contradiction. 
	\end{proof}

	\begin{example}\label{E:Heisenberg}
		For $n\ge 1$, let $\h_{2n+1}$ be the $(2n+1)$-dimensional Heisenberg  Lie $\C$-algebra. It is the $2$-step nilpotent linear Lie $\C$-algebra whose basis is $e_{1,2},\ldots,e_{1,n+2}, \newline e_{2,n+2},\ldots,e_{n+1,n+2}$. The $2$-step nilpotent Lie $\C$-algebra
		\[
		\h_{2n+1} \ltimes \h_{2n+1}^{+}= \left\{ \begin{pmatrix}
		a & x \\ 
		0 & a \\
		\end{pmatrix} \in \mathfrak{gl}_{4n+2}(\C)\ \Big| \ a,x\in \h_{2n+1} \right\}
		\]
		is a finite Morley rank Lie ring without the CBP, by Theorem \ref{T:Lie}. 
	\end{example}

	\subsubsection{Baudisch's Lie algebras} To conclude this section, we point out that  the work of Blossier and Jimenez \cite{BJ22} on  Baudisch's Lie algebras yields some further examples of Lie algebras without the CBP. 
	
	Consider the language of $2$-step nilpotent graded Lie $\F_q$-algebras, that is, the language of $\F_q$-vector spaces with a binary function for the Lie bracket together with two unary predicates standing for the components of the gradation.  Within this language, Baudisch \cite{aB96, aB09} constructed an uncountably categorical $2$-step nilpotent connected Lie  $\F_q$-algebra $\B$  (for $p>2$ and $q$ a $p$-power) of Morley rank $2$ which is graded and whose theory is CM-trivial, we refer to \cite{aP95} for basic properties on CM-triviality.  More precisely, Baudisch proved in \cite[Corollary 8.9]{aB09} the following:
	
	\begin{fact}\label{F:Baudisch} In the language of $2$-step nilpotent graded Lie $\F_q$-algebras, there is a CM-trivial uncountably categorical $2$-step nilpotent connected Lie  $\F_q$-algebra $\B=\B_1 \oplus \B_2$  (for $p>2$ and $q$ a $p$-power) of Morley rank $2$ such that both $\B_1$ and $\B_2$ are the interpretation of the unary predicates, and $\B_1$ is strongly minimal. Furthermore, we have that  $[\B_1,\B_1] =\B_2=\z(\B)$ and $[\B,\B_2]=0$. In particular, the theory of $\B$ is almost strongly minimal.  
	\end{fact}
	
	\begin{remark}\label{R:Baudisch}
	Consider the Baudisch Lie $\F_q$-algebra $\B$ in the language of rings and note that $\B_2=\z(\B)$ is strongly minimal by the fact above. In particular, we have that $\B_2 = [a,\B]$ and that $\z(\B)= \{ x\in \B \ | \ [x,a]=0\}$ for any $a\in \B\setminus \B_2$. Thus, the definable map $x\mapsto [a,x]$ yields that $\B$ is analysable in two steps to $\z(\B)= \B_2$. So, the ring theory of $\B$ is uncountably categorical. On the other hand, notice that in the language of rings $\B$ is not almost strongly minimal, since every $\F_q$-linear map $f:\B_1 \to \B_2$ induces an automorphism $a_1 + a_2 \mapsto a_1 + (a_2 + f(a_1))$ of $\B=\B_1 \oplus \B_2$ which fixes pointwise $\B_2$ and the kernel of $f$.
	\end{remark}

	In \cite[Section 3]{BJ22} Blossier and Jimenez studied a family of derivations of $\B$ in order to construct an additive cover of $\B$ without the CBP, by adapting the strategy followed in \cite{HPP13}. They construct many derivations on Baudisch's Lie algebra by means of bilinear algebra. We isolate the crucial point of their construction in the following statement, which we extracted from the proof of \cite[Corollary 4.6]{BJ22}, see the fourth paragraph on page $52$.
	
	\begin{fact}\label{F:BJ}
		In the language of $2$-step nilpotent graded Lie $\F_q$-algebras consider  a $2$-step nilpotent graded Baudisch Lie algebra $\B=\B_1 \oplus \B_2$ over $\F_q$ with $p>2$ and $q$ a $p$-power. Let $a,b_1,\ldots,b_n$ be independent elements of $\B_1$. For every $e\in\B_1\setminus \{0\}$, if $e$ is independent from $a,b_1,\ldots,b_n$, then there is a derivation $\delta:\B\to\B$ vanishing at $b_1,\ldots,b_n$ and such that $\delta(\B_1)\subset \B_1$ with $\delta(a)=e$.
	\end{fact}

	Using this result of Blossier and Jimenez, we can follow the lines of their proof and adapt the proof of Theorem \ref{T:Lie} to show that $\B\ltimes \B^{+}$ does not satisfy the CBP. Note that $\B\ltimes \B^{+}$ is definable in $\B$, so it is CM-trivial as well by N\"ubling result \cite[Theorem 4.2]{hN05}. In particular, it does not interpret a field by \cite[Proposition 3.2]{aP95}. 
	
	\begin{theorem}\label{T:Lie_B}
		Let $\B$ be the $2$-step nilpotent Baudisch Lie $\F_q$-algebra. The $2$-step nilpotent Lie ring $\B\ltimes \B^{+}$ has finite Morley rank and its theory is CM-trivial but does not satisfy the CBP.
	\end{theorem}
	\begin{proof}
		Set $\g=\B\ltimes \B^{+}$ and note that it is the triangular matrix ring $\Lambda(\B,\B^+)$ corresponding to $\B$ and $\B^+$. Let $T_0$ denote the Lie ring theory of $\g$ and let $T$ be the theory of $\g$ with all the induced structure from $\B$, in the language of $2$-step nilpotent graded Lie $\F_q$-algebras. It is clear that $T_0$ is a reduct of $T$ and that both have finite Morley rank. The subset $\B_1\subset \B$ is a definable group in $T$, as so is $\B$. Moreover, we know by Baudisch's construction (see Fact \ref{F:Baudisch}) that $\B_1$ is strongly minimal. Along the proof there will be an interplay between these to theories, so to distinguish the distinct algebraic closures, we denote by $\acl_0$ the algebraic closure in the sense of $T_0$ and by $\acl$ the one in the sense of $T$. Likewise, we adapt the notation for canonical bases.
		
		Suppose, to get a contradiction, that the Lie ring theory $T_0$ of $\g$ satisfies the CBP. So, by Lemma \ref{L:Analysis} we obtain that $\g/\z(\g)$ is almost internal to a definable, over a finite set $A_0$, subgroup $H$ of $ \{0\}\times \B$, and note that 
		\[
		\z(\g) = \z(\B) \times \z(\B) = \B_2\times \B_2.
		\] 
		In fact, note by the remark after Lemma \ref{L:Analysis} that $H$ is an infinite subgroup of $\{0\}\times \B_2$, so it is $\{0\}\times \B_2$, which is definable without parameters.
		
		Let $A$ be a finite subset of $\g$ containing $A_0$ and exemplifying that $\g/\z(\g)$ is almost internal to $H$. Fix a countable  model $\g_0 \preceq \g$, in the sense of $T$, containing $A$ and choose a generic element $a\in \B_1\subset \B$ over $\g_0$. Set $\alpha=(a,0)\in \g$. It then follows that $\alpha + \z(\g)\in \acl_0(\g_0,\bar c)$ for some finite tuple $\bar c$ of elements from $H$, so we have that
		\[
		\alpha+\z(\g)\in \acl_0(\cb_0(\alpha,\bar c/\g_0),\bar c).
		\] 
		Let  $(\alpha_i,\bar c_i)_{i\le n}$ be an initial segment of a Morley sequence of $\tp(\alpha,\bar c/\g_0)$, in the sense of $T$, with $\alpha_0=\alpha$ and $\bar c_0=\bar c$. By \cite[Lemma 2.1]{BMPW15}, this is also a Morley sequence in the sense of $T_0$ of $\tp_0(\alpha,\bar c/\g_0)$ and thus $\cb_0(\alpha,\bar c/\g_0)\in\dcl_0((\alpha_i,\bar c_i)_{1\le i\le n})$. So  
		\[
		\alpha+\z(\g)\in \acl_0((\alpha_i)_{1\le i\le n}, (\bar c_i)_{ i\le n}).
		\]
		Observe by $A_0$-invariance that $\bar c,\bar c_1,\ldots,\bar c_{n}$ are finite tuples from $H$. Also, each $\alpha_i=(a_i,0)$ for some $a_i\in \B_1$ and $a,a_1,\ldots,a_n\in \B_1$ are independent, since $a\in\B_1$ was taken to be generic over $\g_0$. Therefore, for every non-zero $e\in\B_1$  independent from $a,a_1,\ldots,a_n$, we can find  by Fact \ref{F:BJ} a non-trivial derivation $\delta_e$ on $\B$ vanishing at $a_1,\ldots,a_n$ such that $\delta_e(a)=e$. 
		
		From now on, we work within the reduct $T_0$ of $T$, and recall that $\g=\Lambda(\B,\B^+)$. By Lemma \ref{L:Aut} there exists a Lie ring automorphism $\sigma_e$ of $\g$ associated to $\delta_e$. So, we have that $\sigma_e(\alpha)=\sigma_e((a,0)) = (a,e)$ and that $\sigma_e$ fixes pointwise $\{0\}\times \B$ and $\alpha_1=(a_1,0),\ldots,\alpha_n=(a_n,0)$, by construction. Choosing infinitely many distinct elements $e_i\in \B_1$ for $i$ in $\N$ such that  each $e_i$ is independent from $a,a_1,\ldots,a_n$, we see for $i\neq j$ that $e_i - e_j \not\in \B_2$ since $\B_1\cap \B_2=\{0\}$,  so
		\[
		\sigma_{e_i}(\alpha) - \sigma_{e_j}(\alpha) = (a,e_i) - (a,e_j) = (0,e_i-e_j)  \not\in \z(\g).
		\] 
		Hence, the element $\alpha + \z(\g)$ is not $0$-algebraic over $\alpha_1,\ldots,\alpha_n,\bar c,\bar c_1,\ldots,\bar c_n$, which yields the desired contradiction. 
	\end{proof}
	
	\section{From Lie rings to groups}\label{s:Group}
	
	In this short section we show the existence of $2$-step nilpotent groups of finite Morley rank without the CBP.
	
	Let $L$ be a connected Lie ring of finite Morley rank which is $2$-step nilpotent and such that it is $2$-torsion-free and $2$-divisible. Given $x\in L$, we denote by $\frac{1}{2}x$ the unique element $y\in L$ such that $y+y=x$.
	
	The Baker-Campbell-Hausdorff formula yields a group operation $*$ on the set $L$:
	\[
	x* y := x + y + \frac{1}{2}[x,y],
	\]
	where the identity element is $0$ and the inverse of an element $x\in L$ is $-x$. So, the group $\mathcal{G}(L)=(L,*)$ is clearly definable in the Lie ring structure $L$. Furthermore, an easy computation yields that the group commutator is precisely the Lie bracket:
	\[
	[x,y]_{\mathcal{G}(L)} :=x*y *(-x) *(-y) = [x,y].
	\]
	Thus, the domain of the center and the commutator of the group are the center and the commutator of the Lie ring, respectively. Denote by $nx$ the product $x*x* \ldots* x$ where the symbol $*$ appears $n-1$ times. By induction on $n$ it is straightforward to see that 
	\[
	x * \stackrel{n}{\ldots} * x =  x + \stackrel{n}{\ldots} + x.
	\]
	So, the group $\mathcal G(L)$ is $n$-torsion-free and hence $n$-divisible, if so is $L$. In particular, it is $2$-divisible. Also, the element $\frac{1}{2}x$ is the unique element $y\in L$ such that $y*y = x$.
	
	On the other hand, we also have that the Lie ring structure is definable from the group structure. Namely, the Lie bracket and the additive operation are defined as 
	\[
	[x,y] := [x,y]_{\mathcal{G}(L)}  \ \text{ and } \ x+y:= x*y * (-\frac{1}{2}[x,y])
	\]
	respectively. In particular, the groups of automorphism of the structures $(\mathcal G(L),*)$ and $(L,+,[\cdot\, , \cdot])$ are equal.  
	Altogether we deduce:
	
	\begin{lemma}
		Let $L$ be a connected Lie ring of finite Morley rank such that $L/Z(L)$ is not almost $\Pp$-internal. Assume further that $L$ is $2$-torsion-free. If $L$ is $2$-step nilpotent, then the group $G=\mathcal G(L)$ has finite Morley rank and $G/Z(G)$ is not almost $\Pp$-internal.
	\end{lemma}
	\begin{proof}
		Suppose on the contrary that $G/Z(G)$ is almost $\Pp$-internal. Let $A\subset G$ be a finite set and let $X_1,\ldots,X_n$ be strongly minimal sets, defined over $A$, witnessing this. Note that since the Lie ring structure of $L$ is definable in $G$, each set $X_i$ is strongly minimal in $L$. Also, the set $Z(G)$ is precisely $Z(L)$ and so
		\[
		x*(-y) \in Z(G) \ \Leftrightarrow \ x-y\in Z(L),
		\]
		by the previous discussion. Hence, we deduce that $L/Z(L)$ is almost internal to $X_1,\ldots,X_n$, which yields that it is almost internal to $\Pp$, a contradiction.  
	\end{proof}
	
	Altogether, we deduce the existence of groups without the CBP:
	
	\begin{cor}\label{C:CBP-Group}
		There is a $2$-step nilpotent (CM-trivial) group $G$ of finite Morley rank whose theory in the pure language of groups does not satisfy the CBP.
	\end{cor}
	\begin{proof}
		Combining the lemma above with Theorem \ref{T:Lie} applied to a Heisenberg Lie $\C$-algebra (see Example \ref{E:Heisenberg}) we obtain the statement without CM-triviality. On the other hand, we could also combine the lemma above with Theorem \ref{T:Lie_B} to obtain a $2$-step nilpotent  group $G$ of finite Morley rank interpretable in a CM-trivial theory. Hence, N\"ubling's result \cite{hN05} yields that this group is also CM-trivial.
	\end{proof}
	
	Observe that the proof above shows that there is an algebraic group over the field of complex numbers without the CBP. On the contrary, it is conjectured that every algebraic group over an algebraically closed field of positive characteristic satisfies the CBP, see \cite[p. 874]{HPP13}.

	\section{Commutative unitary rings}\label{s:Rings}
	
	In this section we analyse whether commutative unitary rings of finite Morley rank satisfy the CBP. This family of rings was characterised by Cherlin and Reineke in \cite[Theorem 3.1]{CR76} as follows:
	
	\begin{fact}\label{F:CR-FMR}
	Let $R$ be  commutative unitary ring. The ring $R$ has finite Morley rank if and only if it is isomorphic to a direct product $R_1\times \ldots\times R_k$ of commutative local rings $R_i$ of finite Morley rank whose unique maximal ideal $\m_i$ is nilpotent, and whose residue field $R_i/\m_i$ is finite or algebraically closed.  
	\end{fact}

 	In particular, a  commutative ring of finite Morley rank has only finitely many maximal ideals. We denote by $\Kk(R)$ the finite family of residue fields of $R$.
 	
 	Next, to ease notation, we reformulate part of Fact \ref{F:CR-FMR} empathising on definability aspects. In fact, these are implicitly given in the proof of \cite[Theorem 3.1]{CR76}. We give a direct proof for completeness.
 	
	\begin{lemma}\label{L:Def}
	Let $R$ be commutative unitary ring of finite Morley rank. For each field $K_i$ in $\Kk(R)$, there is an definable ring $R_i$ whose residue field is definably isomorphic to $K_i$ and $R$ is definably isomorphic to $R_1\times \ldots \times R_k$. 
\end{lemma}
\begin{proof}
By Fact \ref{F:CR-FMR} we may assume that $R=R_1'\times \ldots\times R_k'$ where each $R_i'$ is a local ring whose maximal ideal is $\m_{i}'$. Set $u_i\in R$ to be the element with $1$ in the $i$th coordinate and $0$ elsewhere. The ring homomorphism $a\mapsto a\cdot u_i$ from $R$ to $R$ has as an image the subring 
\[
\{0\}\times\ldots \times \{0\} \times R'_{i} \times \{0\}\times \ldots \times \{0\},
\]
which is definable. Thus, any maximal ideal $\m_i$ of $R$, which is of the form 
\[
R'_1\times\ldots \times R'_{i-1} \times \m'_{i} \times R'_{i+1}\times \ldots \times R'_k,
\]
is definable since $\m'_{i}$ is the set of non-units of $R'_i$. Let $\m_1, \ldots,\m_k$ be the maximal ideals of $R$ and let $K_i=R/\m_i$. As each ideal $\m_{i}'$ is nilpotent there exists some $n\ge 1$ such that $\m_{i}'^n=\{0\}$ for every $i$. Consequently, we have that $\m_i^n$ is definable and $R/\m_i^n$ is definably isomorphic to the definable subring of $R$ corresponding to $R'_i$. Setting $R_i=R/\m_i^n$ we obtain the statement. 
 \end{proof}
For local rings with an infinite residue field we also have the following result \cite[Lemma 3 and 6]{aC77}. See also \cite[Corollary A.7]{BN94}.
\begin{fact}\label{F:Def}
Let $R$ be a commutative unitary local ring of finite Morley rank and assume that its residue field is infinite. Then, every non-null ideal of $R$ is infinite and finitely generated, so definable.
\end{fact}

The statement above might not hold if the residue field is finite. 

\begin{example}\label{Ex:1}
Consider the ring of polynomials $\F_p[(y_i)_{i\in \R}]$ with coefficients in $\F_p$ and on infinitely many indeterminates. Let $R=\F_p[(y_i)_{i\in \R}]/(y_iy_j)_{i,j}$. It follows that $R$ is a commutative unitary local ring which is  totally categorical, since the product is definable from the $\F_p$-vector space structure on $R$. The maximal ideal $\m$ of $R$ is not finitely generated. Also, any proper infinitely generated ideal, other than $\m$, is not definable. 
\end{example}
In general, we will see that the situation drastically differs depending on the infinitude of the residue field, as well as the characteristic of the ring. Consequently, we shall distinguish these two distinct cases. 

\subsection{Infinite residue field} Local commutative unitary rings of finite Morley rank are  uncountably categorical whenever the residue field is infinite, see \cite[Theorem 3.2 and 3.3]{CR76} and \cite{aC77}. The following yields a model-theoretic proof of this, and provides some further information. 
			
\begin{lemma}\label{L:CR-Local}
Let $R$ be a commutative unitary local ring of finite Morley rank, possibly with additional structure. If its residue field $K$ is infinite, then the field $K$ is an almost strongly minimal definable set and the maximal ideal $\m$ is $K$-internal. In particular, the ring $R$ is $K$-analysable in at most two steps. 
\end{lemma}
\begin{proof} Let $\m$ denote the maximal ideal of $R$, which is nilpotent by Fact \ref{F:CR-FMR}. Let $n\ge 1$ be the least natural number such that $\m^n=\{0\}$ and note that each ideal $\m^i$ is definable by Fact \ref{F:Def}. Thus, each definable quotient $\m^i/\m^{i+1}$ is a finite dimensional vector space over the definable infinite field $K=R/\m$, with scalar multiplication given by $\bar k\cdot (x + \m^{i+1})= k\cdot x + \m^{i+1}$. So, since $R$ has
	finite Morley rank, these vector spaces must be finite dimensional. Hence, we see that $\m^i$ is $\{\m^{i+1},K\}$-internal, which yields recursively on $1\le i<n$ that $\m$ is $K$-internal, so $R$ is $K$-analysable in two steps. On the other hand, as $K$ is an infinite field definable in a finite Morley rank structure, it is almost strongly minimal by \cite[Lemma 3]{fW01}.\end{proof}

As a corollary, we immediately obtain the following:
\begin{cor}\textnormal{(\cf \cite[Theorem 3.2 and 3.3]{CR76})}
Let $R$ be a commutative unitary local ring of finite Morley rank with an infinite residue field. Its theory is uncountably categorical and its residue field is the unique  almost strongly minimal definable set up to non-orthogonality.
\end{cor}\begin{proof} Assume that $R$ is saturated. By Lemma \ref{L:CR-Local}, the infinite residue field is the unique almost strongly minimal definable set up to non-orthogonality, so $R$ is unidimensional and consequently it is uncountably categorical. 
\end{proof}

Before proceeding and stating the main results concerning commutative (local) rings with infinite residue fields, we clarify that in this context the CBP is equivalent to almost $\Pp$-internality:

\begin{lemma}\label{L:Equivalence}
Let $R$ be a saturated commutative unitary ring of finite Morley rank and assume that every field in $\Kk(R)$ is infinite. Then, the fields in $\Kk(R)$ are the unique almost strongly minimal sets of $R$ up to non-orthogonality. Furthermore, the following are equivalent:
\begin{enumerate}
	\item The theory of $R$ satisfies the CBP.
	\item The ring $R$ is almost $\Kk(R)$-internal.
	\item Every definable (interpretable) commutative unitary local ring in $R$ with an infinite residue field is almost strongly minimal in the theory of $R$.
\end{enumerate} 
In particular, if $R$ is a local ring, then the CBP holds if and only if $R$ is almost strongly minimal with respect to its residue field.
\end{lemma}
\begin{proof}
By Lemma \ref{L:Def}, the ring $R$ is definably isomorphic to $R_1\times \ldots \times R_k$, where each $R_i$ is a definable commutative unitary local ring whose residue field is definably isomorphic to some field $K_i\in \Kk(R)$. By Lemma \ref{L:CR-Local}, the infinite field $K_i$ is almost strongly minimal (and non locally modular). Furthermore, every ring $R_i$ is $K_i$-analysable in two steps, so $R$ is $\Kk(R)$-analysable in two steps as well. This implies that $K_1,\ldots,K_k$ are the unique almost strongly minimal sets up to non-orthogonality. Hence, being almost $\Pp$-internal and almost $\Kk(R)$-internal are equivalent. This shows the first part of the statement.

Now, the implication that (2) implies (1) is  obvious. Also, assuming (3) we have by the first part of the statement that each ring $R_i$ is almost $\Kk(R)$-internal, so $R$ is almost $\Kk(R)$-internal. 

It remains to prove that (1) implies (3). Suppose that $R$ satisfies the CBP and let $R_0$ be a definable commutative unitary local ring with an infinite residue field. By Zilber's Indecomposability Theorem (see \cite[Lemma 1]{aN89}) the subgroup $R_0^\circ$ is an ideal, so $R_0$ is connected by assumption. Hence, Proposition \ref{P:Ann} applied to $R_0$ yields that $R_0$ is almost $\Kk(R)$-internal. Also, by Lemma \ref{L:CR-Local}, the ring $R_0$ is analysable to its residue field $K$, which is almost strongly minimal. So, $R_0$ is almost strongly minimal by Fact \ref{F:Internal}.
\end{proof}

In view of Lemma \ref{L:Equivalence}, we shall focus on commutative unitary local rings of finite Morley rank. We treat first the characteristic $0$ case, where we see that there are plenty of non almost strongly minimal structures. In fact, the first natural occurring example of an uncountably categorical but not almost strongly minimal structure, due to Podewski and Reineke \cite{PR74}, is the local ring $\C[x]/(x^2)$. Hence, the next theorem generalises their result.

\begin{theorem}\label{T:Charac0}
Let $R$ be a commutative local ring with a unit and of finite Morley rank. Assume that it has characteristic $0$. If $R$ is not a field, then it is not almost strongly minimal nor has the CBP.
\end{theorem}
\begin{proof}
Let $R$ be a commutative local unitary ring of finite Morley rank, which has characteristic $0$. Assume, as we may, that $R$ is uncountable and hence saturated.
 
We see that $R$ is indeed a triangular ring. By Fact \ref{F:CR-FMR}, the maximal ideal $\m$ of $R$ is nilpotent, so the residue field $K=R/\m$ must have characteristic $0$ as well. By Cohen Structure Theorem \cite[Theorem 9]{iC46} there exists a subfield $k\subset R$ which maps isomorphically onto $K$ via the natural homomorphism $k\hookrightarrow R \twoheadrightarrow K$. It is clear that $k\cap \mathfrak{m}=\{0\}$. Moreover, given an arbitrary element $a\in R$, there exists some $b\in k$ such that $a+\m  = b+\m$, so  $a=(a-b) + b \in\mathfrak{m}+ k$ and therefore $R=k\oplus \m$. As a consequence, we obtain that $R$ is (isomorphic to) the triangular ring $\Lambda(k,\m)$. Furthermore, the ring $\m$ is a vector space over $k$ and nilpotent. Thus, Theorem \ref{T:Algebra} yields that $R\cong \Lambda(k,\m)$ does not have the CBP nor is almost strongly minimal, by Lemma \ref{L:Equivalence}.
\end{proof}

\begin{remark}
	Note that the proof above gives also an alternative proof to the one given in \cite{PR74} for the local ring $\C[x]/(x^2)$.
\end{remark}
Now, we consider the positive characteristic case. Before proceeding we briefly recall the notion of {\em multiplicative representative}.
\begin{definition}
Let $R$ be a commutative unitary local ring with a residue field $K$ of characteristic $p>0$. A {\em multiplicative representative} of an element $\alpha\in K$ is an element $a\in R$ such that $\overline a=\alpha$ and for every $n\ge 1$ there exists some $b_n\in R$ such that $a=b_n^{p^n}$.
\end{definition}
The following is \cite[Lemma 9]{aC77}, see also \cite[Lemma 7]{iC46} or \cite[31.3]{mN62}.
\begin{fact}\label{F:MultRep}
Let $R$ be a commutative unitary local ring with a perfect residue field $R/\m$ of characteristic $p>0$ with  $\m^n=\{0\}$ for some $n\ge 1$. Every element of $R/\m$ has a unique multiplicative representative and the set $X$ of multiplicative representatives is 
\[
X=\left\{ a\in R \ | \ a= b^{p^{n}} \text{ for some $b\in R$}\right\}.
\]
\end{fact}
Now, we prove the following:
\begin{theorem}\label{T:Charac+}
Let $R$ be a commutative unitary ring of finite Morley rank and assume that $R$ has positive characteristic and that every field in $\Kk(R)$ is infinite. Then $R$ is almost $\Kk(R)$-internal.
In particular, if in addition $R$ is a local ring, then it is almost strongly minimal.
\end{theorem}
\begin{proof}
We may assume that $R$ is saturated. By Lemma \ref{L:Equivalence}, we can further assume that $R$ is a commutative unitary local ring of finite Morley rank with an infinite residue field $K=R/\m$. Moreover, the ring $R$ has characteristic a prime power $p^s$, since the maximal ideal $\m$ is nilpotent, so $K$ has characteristic $p>0$. 

We prove that $R$ is $K$-internal. Consider the set $X$ of multiplicative representatives of $R$, which is clearly definable by Fact \ref{F:MultRep}. Moreover, for every $a\in R$ there is a unique multiplicative representative $b_a\in X$ for $\overline a\in R/\m=K$. In particular, we have that $a-b_a\in \m$ by definition. So, as $a$ was arbitrary, we get $R\subset X+\m$, which implies that $R\subset \dcl(X,\m)$. We already know by the proof of Lemma \ref{L:CR-Local} that $\m$ is $K$-internal. So, it suffices to see that $X$ is $K$-internal as well. But this is immediate by uniqueness of the multiplicative representative, since every automorphism of $R$ fixing pointwise $K=R/\m$ must fix pointwise the set $X$. This shows that $X\subset \dcl(K)$, so $R$ is $K$-internal. \end{proof}

Altogether, Theorems \ref{T:Charac0} and \ref{T:Charac+} combined with Lemma \ref{L:Equivalence} yield: 

\begin{cor}\label{C:Charac}
Let $R$ be a commutative unitary ring of finite Morley rank and assume that every residue field is infinite. The theory of $R$ satisfies the CBP if and only if $R$ is a direct product $R_1\times \ldots \times R_n$ where each $R_i$ is a field or a local ring of positive characteristic. 
\end{cor}

Therefore, by the work of Cherlin and Reineke \cite[Theorem 3.3]{CR76} (see also \cite[Corollary 3.5]{CR76}) we immediately obtain:
\begin{cor}
Let $R$ be a commutative unitary Noetherian local ring, which is not a field, with a nilpotent maximal ideal and an algebraically closed residue field of characteristic $0$. Then $R$ is an uncountably categorical structure without the CBP.
\end{cor}  
In particular, rings of  the form $\C[x_1,\ldots,x_n] / \A$, where $\A$ is a primary ideal with radical $(x_1,\ldots,x_n)$, are uncountably categorical without the CBP.  

\subsection{Finite residue field} The family of commutative unitary local rings of finite Morley rank with a finite residue field remains less studied than the previous ones and there is no reasonable algebraic description.

An archetypical example of an uncountably categorical commutative unitary local ring with a finite residue field is the ring  $R=\F_p[(y_i)_{i\in \R}]/(y_iy_j)_{i,j}$ given in Example \ref{Ex:1}. This ring is totally categorical and strongly minimal. So, it is one-based and consequently its theory satisfies the CBP.  Nonetheless, we will see more exotic examples and that most commutative unitary local rings with a finite residue field are not almost strongly minimal.

\subsubsection{Some characterisations} We analyse uncountably categorical commutative local rings of equicharacteristic and with a finite-index connected maximal ideal. That is, uncountably categorical commutative local rings of the form $\F_q\oplus \m\cong \Lambda(\F_q,\m)$, where $q$ is a prime power and $\m= \m^\circ$ is the maximal ideal.

 We begin with two lemmata which will help us to construct ring automorphisms. These are in a way refinements of Lemma \ref{L:Aut} to this framework.

\begin{lemma}\label{L:Auto1}  Let $R$ be a commutative unitary local ring with maximal ideal $\mathfrak m$ such that $R=\mathbb F_q \oplus \mathfrak m$. 	
Let $g:\mathfrak m\rightarrow \mathfrak m$ be a bijective $\F_q$-linear map such that for every $a$ in $\mathfrak m$ the difference $g(a)-a$ lies in $\ann(\mathfrak m)$. If $g$ is the identity on $\mathfrak m^2$, then
	the map 
	\[
	\sigma:R\rightarrow R, \ \sigma(\alpha+a)=\alpha+g(a)
	\] 
	is an automorphism of the ring $R$.
\end{lemma}
\begin{proof}
	Note that $\sigma$ is $\F_q$-linear and a bijection. It remains to show that $\sigma$ is multiplicative. Let $\alpha + a$ and $\beta + b$ be two elements in $R$, where $\alpha,\beta\in \F_q$ and $a,b\in \m$. We have that
	\begin{align*}
	\sigma(\alpha+a)\sigma(\beta+b) & =\big(\alpha+g(a)\big)\big(\beta+g(b)\big)=\alpha\beta+\alpha g(b) +\beta g(a)+g(a)g(b)
	\\ & =\alpha\beta+g(\alpha b + \beta a)+ (g(a)-a+a)g(b) \\ &= \alpha\beta+g(\alpha b + \beta a)+ 	(g(a)-a)g(b)+a(g(b)-b)+ab  
	\\ & =\alpha\beta+g(\alpha b + \beta a)+ ab
	=\alpha\beta + g(\alpha b + \beta a + ab) \\ & =\sigma\big( (\alpha+a) (\beta + b) \big),
	\end{align*}
as desired.
\end{proof}

\begin{lemma}\label{L:Auto2} Let $R$ be a commutative unitary local ring with maximal ideal $\mathfrak m$ such that $R=\mathbb F_q \oplus \mathfrak m$. 
For every $\F_q$-linear map $f:\mathfrak m / \ann(\mathfrak m)\rightarrow \ann(\mathfrak m)$ vanishing on $(\mathfrak m^2 + \ann(\mathfrak m))/\ann(\mathfrak m)$, the map 
\[
\sigma:R\rightarrow R, \ \sigma(\alpha+a)=\alpha+a+f\big(a+\ann(\mathfrak m)\big)
\]
 is an automorphism of the ring $R$.
\end{lemma}
\begin{proof} 
	Let 
	$g:\mathfrak m\rightarrow \mathfrak m$ be the map defined by $g(x)=x+f\big(x+\ann(\mathfrak m)\big)$. Note that $g$ is $\F_q$-linear and is the identity on $\mathfrak m^2$ since $f$ vanishes on $\big(\mathfrak m^2 + \ann(\mathfrak m)\big)/\ann(\m)$. Furthermore, for every $a\in\m$ the difference $g(a)-a=f\big(a+\ann(\mathfrak m)\big)$ lies in $\ann(\mathfrak m)$. By Lemma \ref{L:Auto1} we need only show that $g$ is a bijection, since $\sigma(\alpha+a)=\alpha+g(a)$. 
	If $g(a)=a+f\big(a+\ann(\mathfrak m)\big)=0$, then $a\in \ann(\m)$ since  $f(a+\ann(\m))\subset \ann(\m)$. Therefore $f\big(a+\ann(\mathfrak m)\big)=f\big(0+\ann(\mathfrak m)\big)=0$, thus 
	\[
	0=g(a)=a+f\big(a+\ann(\mathfrak m)\big)=a.
	\] Hence, the map $g$ is injective. Note that $g$ maps the element $a-f\big(a+\ann(\mathfrak m)\big)$ to $a$, so we deduce that $\sigma$ is a bijection, as desired.
\end{proof}

Next, we characterise almost strongly minimal equicharacteristic local rings  with a connected finite-index maximal ideal.

\begin{prop}\label{P:AlmostStronglyMinimal}
Let $R$ be an uncountably categorical commutative unitary local ring whose maximal ideal $\mathfrak m$ is connected and $R=\mathbb F_q \oplus \mathfrak m$. Then, the ring $R$ is almost strongly minimal if and only if $\mathfrak \m=\ann(\m)$.
\end{prop}
\begin{proof} Observe first that for every  natural number $k\ge 1$ the ideal $\mathfrak m^k$ is definable and connected by  \cite[Lemma 7]{aN89}, as an application of Zilber's Indecomposability Theorem. Thus, as $\m$ is nilpotent, we deduce that $\ann(\m)$ contains some $\m^r\neq \{0\}$ and so it is infinite.

Now, if $\ann(\mathfrak m)=\mathfrak m$, then $\m$ is clearly an $\F_q$-vector space. Moreover, notice that the ring multiplication of $R$ is definable from the multiplication of $\mathbb F_q$ and the vector space structure, so $R$ is in fact the structure $\F_q\times\mathfrak m$ with the field structure on the first coordinate and the $\F_q$-vector space structure on the second one.  This implies that the induced structure of $R$ in $\mathfrak m$ is the one given by the $\mathbb F_q$-vector space structure, so  $\mathfrak m$ is a strongly minimal subset of $R$. Hence, we deduce that $R$ is almost strongly minimal.

	For the other direction, suppose that $R$ is almost strongly minimal. We assume that $\ann(\mathfrak m)$ is a proper subset of $\mathfrak m$ and aim for a contradiction. Because $\mathfrak m$ is connected, the quotient $\mathfrak m/\ann(\mathfrak m)$ must be infinite. Since $R$ is almost strongly minimal, we deduce that $R$ is almost internal to $\mathfrak m/\ann(\mathfrak m)$. Therefore, there are elements $a_1,\ldots,a_n$ in $\mathfrak m$ such that $R$ is contained in $\acl\big(a_1,\ldots,a_n,\mathfrak m/\ann(\mathfrak m)\big)$. 
	
	We now show that also the subgroup $\mathfrak m^2 + \ann(\mathfrak m)$ must have infinite index in $\mathfrak m$. Otherwise $\mathfrak m= \mathfrak m^2 + \ann(\mathfrak m)$ and it would follow that 
	$\mathfrak m^2 = \big( \mathfrak m^2 + \ann(\mathfrak m)\big)^2=\mathfrak m^4$, which in turn would imply that $\mathfrak m^2 =\{0\}$, since $\mathfrak m$ is nilpotent. So, we would get that $\mathfrak m= \mathfrak m^2 + \ann(\mathfrak m) =\ann(\m)$, contradicting our assumption. Hence, we deduce that the subgroup $\mathfrak m^2 + \ann(\mathfrak m)$ has infinite index in $\mathfrak m$, as claimed.
	
	Since
	\[
	{\big({\mathfrak m}/{\ann(\mathfrak m)}\big)} \big/ {\big({\mathfrak m^2 + \ann(\mathfrak m)}/{\ann(\mathfrak m)}\big)}\cong \mathfrak m \big/ \big(\mathfrak m^2 + \ann(\mathfrak m)\big),
	\]
	we deduce that the $\mathbb F_q$-vector space $\mathfrak m/\ann(\mathfrak m)$ is infinite dimensional over the subspace $\big(\mathfrak m^2 + \ann(\mathfrak m)\big)/\ann(\mathfrak m)$. Therefore, there exists an element $a\in \mathfrak m$ such that the coset $a+\ann(\mathfrak m)$ does not lie in the $\mathbb F_q$-vector space generated by 
	\[
	\big(\mathfrak m^2 + \ann(\mathfrak m)\big)/\ann(\mathfrak m), a_1 + \ann(\mathfrak m),\ldots,a_n + \ann(\mathfrak m).
	\] 
	Hence, we can choose for each $b\in\ann(\mathfrak m)$ an $\F_q$-linear map \[f_b:\mathfrak m / \ann(\mathfrak m)\rightarrow \ann(\mathfrak m)\] which vanishes on  $\big(\mathfrak m^2 + \ann(\mathfrak m)\big)/\ann(\mathfrak m), a_1 + \ann(\mathfrak m),\ldots,a_n + \ann(\mathfrak m)$ and maps $a+\ann(\mathfrak m)$ to $b$. By Lemma \ref{L:Auto2}, the function
	$\sigma_b :R\rightarrow R$ defined by 
	\[
	\sigma_b(\gamma+c)=\gamma+c+f_b\big(c+\ann(\mathfrak m)\big)
	\]
	is  a ring automorphism, where $\gamma\in\F_q$ and $c\in\m$. Note that $\sigma_b$ fixes $a_1,\ldots,a_n$ and $\mathfrak m/\ann(\mathfrak m)$ pointwise, and maps $a$ to $a+b$. 
	This shows that $a$ has $|\ann(\m)|$ many conjugates over $a_1,\ldots,a_n,\mathfrak m/\ann(\mathfrak m)$, which is the desired contradiction because $R$ is contained in $\acl(a_1,\ldots,a_n,\mathfrak m/\ann(\mathfrak m))$ and $\ann(\m)$ is infinite.
\end{proof}

As a consequence, since every local ring $R=\F_q\oplus \m$ whose maximal ideal $\m$ annihilates itself is interpretable in an $\F_q$-vector space structure, we deduce:

\begin{cor}
Let $R$ be an uncountably categorical commutative unitary local ring whose maximal ideal $\mathfrak m$ is connected and $R=\mathbb F_q \oplus \mathfrak m$. If $R$ is almost strongly minimal, then it is totally categorical.
\end{cor}

We point out that in Proposition \ref{P:AlmostStronglyMinimal} the maximal ideal must be connected, otherwise the statement does not hold as witnessed in the following example.

\begin{example}
For $p>0$ prime, consider the ring $S=\F_p[(y_i)_{i\in \R}]/(y_iy_j)_{i,j}$ given in Example \ref{Ex:1} and let $R=S[x]/(x^2)$. This is an $\aleph_0$-categorical ring of finite Morley rank, since it is interpretable in $S$. In fact, it has Morley rank $2$. The maximal ideal $\m$ of $R$ has index $p$  and consists of elements 
\[
\sum_{i\in \R} \alpha_i y_i + \left( \beta + \sum_{i\in \R} \beta_i y_i\right) x + \A,
\]
where $\A=(x^2, (y_iy_j)_{i,j})$ is the ideal of $R$  generated by $x^2$ and all $y_iy_j$ with $i,j\in\R$, $\beta\in\F_p$ and $\alpha_i,\beta_i\in\F_p$ for $i\in \R$ with all but finitely many zero. Note that $R=\F_p\oplus \m$ and that the connected component $\m^\circ$ of $\m$ has index $p^2$ in $R$ and it is formed by the elements of $\m$ with $\beta=0$. So, we have that $\m^\circ = \ann(\m^\circ)$. Also, the ideal $\m^2$ consists of those elements of $\m$ with $\alpha_i=0$ for $i\in \R$ and $\beta=0$. So, the ideal $\m^2$ is strongly minimal. By considering the map $z\mapsto z\cdot (x+\A)$ from $\m^\circ\to\m^2$, we see that this is a surjective homomorphism with kernel $\m^2$, so $\m^\circ$ is analysable in $\m^2$ in two steps. Hence, we deduce that $R$ is uncountably categorical, see also \cite[Example 6.2]{CR76}. 

However, we see that it is not almost strongly minimal. Suppose otherwise that there is some finite subset $A\subset R$ such that $R\subset \acl(A,\m^2)$. Choose a natural number $k$ large enough such that for $i\in\R$ with $i\ge k$ no indeterminate $y_i$ occurs in the elements of $A$. Given a natural number $\ell \ge k$ we consider the map $\delta_{k,\ell}: \m \to \m$, which is defined as
\[
\sum_{i\in \R} \alpha_i y_i + \left( \beta + \sum_{i\in \R} \beta_i y_i\right) x + \A \mapsto \alpha_k y_\ell + \A.
\] 
This is an $\F_q$-linear map of $\m$ which vanishes on $\m^2\cup A$. So, the function
\[
g_{k,\ell} : \m\to \m , \ a \mapsto a+ \delta_{k,\ell} (a)(x + \A)
\] 
is a bijective $\F_q$-linear function which is the identity on $\m^2$ and such that for every $a\in \m$ the difference $g_{k,\ell}(a) - a = \delta_{k,\ell} (a)(x+\A) \in \m^2 =\ann(\m)$. Thus, by Lemma \ref{L:Auto1}, we obtain for each $\ell\ge k$ an automorphism $\sigma_{k,\ell} \in \Aut(R)$, which is defined as $\sigma_{k,\ell}(\alpha + a) = \alpha + g_{k,\ell}(a)$, for $\alpha\in\F_p$ and $a\in \m$. In particular, each $\sigma_{k,\ell}\in \Aut(R)$ fixes pointwise $\m^2\cup A$, so we deduce that $R\not\subset\acl(A,\m^2)$, a contradiction.
\end{example} 

Next, we characterise equicharacteristic local rings with connected finite-index maximal ideal which have the CBP.

\begin{theorem}\label{T:CBP}
Let $R$ be an uncountably categorical commutative unitary local ring whose maximal ideal $\mathfrak m$ is connected and $R=\mathbb F_q \oplus \mathfrak m$. Then, the ring $R$ has the CBP if and only if the quotient $R/\ann(\m)$ is almost strongly minimal.
\end{theorem}
\begin{proof}
	Assume that $R$ has the CBP. We already know by Proposition \ref{P:Ann} that $\m/\ann(\m)$ is almost strongly minimal, and so is $R/\ann(\m)$.
	
	For the other direction, assume that $R/\ann(\m)$ is almost strongly minimal and
	fix a countable elementary substructure $R_0\preceq R$. Given a tuple $a=(a_1,\ldots,a_n)$ of elements in $R$, we need only show that the type $\stp(\cb(a/R_0)/a)$ is almost $R/\ann(\m)$-internal in order to deduce that $R$ has the CBP.
	
	Choose a formula $\varphi(x,b)$ in $\tp(a/R_0)$ of minimal Morley rank and degree one. 
	Since the multiplicative identity of $R$ lies in $\F_q$, it follows that every element from $\F_q$ is definable over the empty set. Hence, after possibly replacing $a$ and $b$ by interdefinable tuples, we may assume that $a$ and $b$ are tuples of elements in $\m$. 
	Furthermore, after possibly permuting the coordinates, we may assume that
	$(b_{1},\ldots,b_k)$ is the maximal subtuple of $b=(b_{1},\ldots,b_m)$ which is $\F_q$-linearly independent over $\langle \m^2 , a\rangle_{\F_q}$. Hence, for each $k<j\leq m$ there are elements $\xi_{j,i}$ in $\F_q$ such that the difference
	$b_{j} - \sum_{i=1}^{j-1}\xi_{j,i}b_i$ lies in $\langle \m^2 , a\rangle_{\F_q}$.
	Note that the $\F_q$-vector space $\langle a_1 + \m^2, \ldots, a_n + \m^2 \rangle_{\F_q}$ is finite. Choose a finite subset $C\subset R_0$ such that 
	\[
	\{n + \m^2\in \langle a_1 + \m^2, \ldots, a_n + \m^2 \rangle_{\F_q}\ |\ n\in R_0 \}=\{c + \m^2 \ |\ c\in C \}.
	\]
	Given elements $n_1,\ldots,n_k$ in $R_0$, note that they are $\F_q$-linearly independent over $\langle \m^2 , a\rangle_{\F_q}$ if and only if $R_0\models \phi(n_1,\ldots,n_k)$ where
	\[
	\phi(z_1,\ldots,z_k) =  \bigwedge_{\substack{c\in C \\ \zeta_1,\ldots,\zeta_k\in\F_q }} c-\sum_{i=1}^k \zeta_i z_i \notin\m^2.
	\]
	Since the canonical base $\cb(a/R_0)$ is interdefinable with the canonical parameter $\ulcorner \textnormal{d}_p x \varphi(x,y)\urcorner$, it suffices to show that every automorphism of $R$ which fixes $a$ and the quotient $R/\ann(\m)$ pointwise maps $b$ to another realization of $ \textnormal{d}_p x \varphi(x,y)$.
	
	Let $\sigma\in\Aut(R)$ be an automorphism which fixes the tuple $a$ and the quotient $R/\ann(\m)$ pointwise. Note that  $\sigma$ must fix the set $\m^2$ pointwise as well. Indeed, given $x,y\in\m$ write $\sigma(x)=x+x_0$ and $\sigma(y)=y+y_0$ for some $x_0,y_0\in \ann(\m)$, so we have that
	\[
	\sigma(xy)=	\sigma(x)\sigma(y) = (x+x_0)(y+y_0) = xy.
	\]
	Hence $\sigma$ fixes the $\F_q$-vector space $\langle\m^2 , a\rangle_{\F_q}$ pointwise. Since for $k<j\leq m$ the difference
	$b_{j} - \sum_{i=1}^{j-1}\xi_{j,i}b_i$ belongs to $\langle \m^2 , a\rangle_{\F_q}$, we deduce for $k<j\le m$ that
	\[
	\sigma(b_{j}) - \sum_{i=1}^{j-1}\xi_{j,i}\sigma(b_i) = b_{j} - \sum_{i=1}^{j-1}\xi_{j,i}b_i .
	\]
	Furthermore, note that the elements $\sigma(b_1),\ldots,\sigma(b_k)$ are again $\F_q$-linearly independent over $\langle\m^2 , a\rangle_{\F_q}$. Therefore, we have that $R\models\psi\big(\sigma(b)\big)$, where $\psi(y_1,\ldots,y_m)$ is the following formula with parameters in $R_0$:
	\[
	\bigwedge_{i=1}^k \big( y_i - b_i \in\ann(\mathfrak m) \ \land 
	\phi(y_1,\ldots,y_k)  \big) \land \\
	\bigwedge_{j=k+1}^my_{j} - \sum_{i=1}^{j-1}\xi_{j,i}y_i = b_{j} - \sum_{i=1}^{j-1}\xi_{j,i}b_i.
	\]
	Hence, we need to show that $R\models \forall y_1 \ldots \forall y_m 
	\big( \psi(y_1,\ldots,y_m)\rightarrow\textnormal{d}_p x \varphi(x,y_1,\ldots,y_m) \big)$.
	Since $R_0$ is an elementary substructure of $R$, it suffices to show that 
	\[
	R_0\models \forall y_1 \ldots \forall y_m 
	\big( \psi(y_1,\ldots,y_m)\rightarrow\textnormal{d}_p x \varphi(x,y_1,\ldots,y_m) \big).
	\]
	So, let $b'=(b'_1,\ldots,b'_m)$ be a tuple of elements in $R_0$ such that $R_0\models \psi(b')$.
	We aim to show that $R_0\models \textnormal{d}_p x \varphi(x,b')$ or equivalently
	that $R\models \varphi(a,b')$. Hence, it suffices to construct an automorphism of $R$ which fixes the tuple $a$ and maps $b$ to $b'$.
	Since $R\models \psi(b')$, 
	the elements $b'_1,\ldots,b'_k$ are
	$\F_q$-linearly independent over $\langle \m^2 , a\rangle_{\F_q}$ and  for each $1\leq i \leq k$ we have that $b'_i - b_i\in \ann(\m)$. 
	Choose a $\mathbb F_q$-basis $\{c_1,\ldots,c_s\}$ of
	$\langle b_1 - b'_1,\ldots,b_k - b'_k\rangle_{\F_q}$ over $\langle \m^2 , a,b_1,\ldots,b_k\rangle_{\F_q}$ and a $\F_q$-basis $\{c'_1,\ldots,c'_{s'}\}$ of
	$\langle b_1 - b'_1,\ldots,b_k - b'_k\rangle_{\F_q}$ over $\langle \m^2 , a,b'_1,\ldots,b'_k\rangle_{\F_q}$. 
	Note that
	\begin{align*}
	\langle \m^2 , a,b_1,\ldots,b_k,c_1,\ldots,c_s\rangle_{ \F_q} &
	=\langle \m^2 , a,b_1,\ldots,b_k,b'_1,\ldots,b'_k\rangle_{ \F_q}  \\
	 & =\langle \m^2 , a,b'_1,\ldots,b'_k,c'_1,\ldots,c'_{s'}\rangle_{ \F_q},
	\end{align*}
	thus it follows that $s'=s$.
	Now choose a $\F_q$-basis  $(d_i)_{i\in I}$ of $\ann(\m)$ over \[\langle \m^2 , a, b_1,\ldots,b_k,b'_1,\ldots,b'_k\rangle_{ \F_q}\] and then
	a $\F_q$-basis  $(e_j)_{j\in J}$ of $\m$ over \[\langle \ann(\m), \m^2 , a, b_1,\ldots,b_k,b'_1,\ldots,b'_k\rangle_{\F_q}.\]
	Note that 
	\[
	\{b_1,\ldots,b_k\} \cup \{c_1,\ldots,c_s\} \cup \{d_i \ |\ i\in I \} \cup \{e_j \ |\ j\in J \}
	\]
	and
	\[
	\{b'_1,\ldots,b'_k\} \cup \{c'_1,\ldots,c'_s\} \cup \{d_i \ |\ i\in I \} \cup \{e_j \ |\ j\in J \}
	\]
	are both a $\F_q$-basis of $\m$ over $\langle \m^2 , a\rangle_{ \F_q}$.
	Therefore there exists an $\F_q$-linear bijection $g:\m \rightarrow \m$ which is the identity on
	\[
	\langle \m^2 , a\rangle_{\F_q} \cup \{d_i \ |\ i\in I \} \cup \{e_j \ |\ j\in J \}
	\] and maps 
	$b_i$ to $b'_i$ for $1\leq i\leq k$ and $c_j$ to $c'_j$ for $1\leq j \leq s$.
	Since $g$ fixes $\langle\m^2 , a\rangle_{ \F_q}$ pointwise and for $k<j\leq m$ the element
	\[
	b'_{j} - \sum_{i=1}^{j-1}\xi_{j,i}b'_i = b_{j} - \sum_{i=1}^{j-1}\xi_{j,i}b_i 
	\]
	lies in $\langle \m^2 , a\rangle_{\mathbb F_q}$, it recursively follows that
	$g(b_j)=b'_j$. Also, since the differences $g(b_i)-b_i=b'_i - b_i$ and hence also $g(c_j)-c'_j=c'_j - c_j$ lie in $\ann(\m)$ for $1\leq i \leq k$ and $1\leq j \leq s$, it follows that $g(m)-m\in \ann(\m)$ for every $m\in \m$. 
	Therefore, we can now conclude with Lemma \ref{L:Auto1} that the map $\sigma:R\rightarrow R$ defined by $\sigma(\alpha+a)=\alpha+g(a)$ is the desired automorphism.
\end{proof}

As an application of the latter two results we obtain the following example of an uncountably categorical local ring with a finite residue field that satisfies the CBP but it is not almost strongly minimal.

\begin{example}\label{E:ring2}
Fix an uncountable algebraically closed field $F$ of positive characteristic $p>0$ and let $\m$ denote the maximal ideal of the local ring $F[x]/(x^3)$. 
Since $\ann(\m)$ equals the ideal $\{ \alpha x^2 + (x^3)  \ | \ \alpha\in F \}\subsetneq \m$, the local ring $R = \F_p \oplus \m$ is not almost strongly minimal by Proposition \ref{P:AlmostStronglyMinimal}. On the other hand, since $R/\ann(\m)$ is an almost strongly minimal set, the ring $R$ has the CBP by Theorem \ref{T:CBP}. 

We remark that $R$ interprets an infinite algebraically closed field, so it is not one-based. Indeed, let $h:\mathfrak m\rightarrow \mathfrak m$ be the definable map which maps the element $a\in\m$ to the product $a\cdot \big(x+(x^3)\big)$. Given $a,b\in \ann(\mathfrak m)$, choose $a',b' \in \m$ such that $h(a')=a$ and $h(b')=b$. Set $a\odot b=a'\cdot b'$ and note that $a\odot b$ does not depend on the choices made for $a'$ and $b'$. It is easy to verify that the interpretable structure $( \ann(\m) , +, \odot)$ is an algebraically closed field isomorphic to $F$. 

More generally, let $\mathfrak n$ denote the maximal ideal of the local ring $F[x]/(x^n)$ for some fixed natural number $n > 3$. The ring $S=\mathbb F_p \oplus \mathfrak n$ is again not almost strongly minimal. The quotient $S/\ann(\mathfrak n)$ is in contrast to the previous example, when considered as a pure ring, not almost strongly minimal by Proposition \ref{P:AlmostStronglyMinimal}. However, it is not difficult to show that $S/\ann(\mathfrak n)$ is again an almost strongly minimal subset of the ring $S$, so the latter has again the CBP by Theorem \ref{T:CBP}. 
\end{example}

\begin{remark}
	In the proof of Theorem \ref{T:CBP}  we have shown that every canonical base  $\cb(a/R_0)$ belongs to $\acl(a,R/\ann(\m))$, whenever $R/\ann(\m)$ is almost strongly minimal. It is worth noticing that this does not mean that $R$ has the strong CBP in the sense of \cite[Definition 3.1]{PP17}, since the tuple $a$ consists only of elements from $R$ and not imaginaries. In fact, it is not difficult to see that the ring $R=\F_p\oplus \m$ from Example \ref{E:ring2} does not satisfy the strong CBP. 
\end{remark}

\subsubsection{A new example} We finish the paper by constructing a finite Morley rank commutative unitary local ring with $\F_q$ as a residue field and whose theory is CM-trivial but does not satisfy the CBP. 

Let $k$ be an arbitary field. Given a $2$-step nilpotent Lie $k$-algebra $\h$, let $S(\h)$ be the set $\h\times \h$ equipped with addition coordinatewise and multiplication given by
\[
(a_1,a_2)\cdot (b_1,b_2) = \big( 0, [a_1,b_2] + [b_1,a_2] \big).
\]
This operation is commutative and also associative, since $\h$ is a $2$-step nilpotent Lie algebra. Hence, we deduce from linearity of the Lie bracket that $S(\h)$ is a commutative ring. In addition, observe that $S(\h)$ is a $k$-vector space. So, we can consider the triangular ring $\Lambda(k,S(\h))$, which is isomorphic to $k \oplus S(\h)$. Furthermore, observe that $S(\h)$ is a nilpotent ideal of $\Lambda(k,S(\h))$. So, the set $\Lambda(k,S(\h))\setminus S(\h)$ consists of units and hence  $\Lambda(k,S(\h))$ is local with $S(\h)$ as the maximal ideal.

\begin{prop}\label{P:CM}
Let $\B$ be Baudisch's Lie $\F_q$-algebra. The theory of the commutative unitary local ring $R=\Lambda(\F_q,S(\B))$ has finite Morley rank and it  is CM-trivial but not one-based.  
\end{prop} 
\begin{proof}
 The ring $R$ is definable in $\B$, so it has finite Morley rank. In addition, it is CM-trivial by N\"ubling's result \cite{hN05}. However, it is not one-based. For otherwise, a similar argument as in Propositon \ref{P:Ann} yields that $R^\circ/\ann(R^\circ)$ is finite. Since $R^\circ = S(\B)$ and $\ann(S(\B)) = \z(\B) \times \z(\B)$, it then follows that $\B$ is abelian,  contradiction.
\end{proof}

Next, consider the semidirect sum $\B\ltimes \B^+$, which is $\Lambda(\B,\B^+)$. The commutative unitary local ring $\Lambda(\F_q,S(\B\ltimes \B^+))$ is a finite Morley rank ring whose theory is CM-trivial by N\"ubling's result \cite{hN05}, since it is clearly definable in $\B$. To ease notation, we regard it as $\F_q \oplus S(\B\ltimes \B^+)$. In particular, its elements are of the form $\alpha + ((a_1,a_2),(b_1,b_2))$ where $\alpha\in \F_q$ and $a_1,a_2,b_1,b_2\in \B$. 

We give a refinement of Lemma \ref{L:Analysis} for this ring.
\begin{lemma}\label{L:Analysis2}
Let $R=\Lambda(\F_q,S(\B\ltimes \B^+))$. The set  $R/\ann(R^\circ)$ is analysable to the definable subgroup 
\[
\A:=(\{0\}  \times \{0\} ) \times (\{0\}\times \z(\B)) \subset S(\B\ltimes\B^+).
\]  
In particular, if $R$ has the CBP, then $R/\ann(R^\circ)$ is almost internal to $\A$.
\end{lemma}
\begin{proof} Set $S=S(\B\ltimes \B^+)$. Since $S$ is definable in the structure of $\B$ and there, the group $S^+$ is definably isomorphic to the additive group $\bigoplus_{i=1}^4\B^+$, we deduce that $R^\circ = S$, as $\B$ is connected by Fact \ref{F:Baudisch}. Therefore, we need to see that $S/\ann(S)$ is analysable to the subgroup $\A$ of the statement.

Now, given an element $y\in S$, let $\rho_y : S\to S$ be the homomorphism $\rho_y(x) = x\cdot y$. So, we have for $x=((x_1,x_2),(x_3,x_4))$ and $y=((y_1,y_2),(y_3,y_4))$ that
\begin{align*}
\rho_y(x) & =  \big( (0,0) , [ (x_1,x_2),(y_3,y_4)]_\g + [(y_1,y_2), (x_3,x_4)]_\g \big)\\ & = \big( (0,0) , ([x_1,y_3]_\B + [y_1,x_3]_\B , [x_1,y_4]_\B - [y_3,x_2]_\B + [y_1,x_4]_\B - [x_3,y_2]_\B ) \big).
\end{align*} 
Fix some element $z\in \B\setminus \z(\B)$ and note that $\z(\B) = \{ x\in \B \ | \ [x,z] = 0\}$ by Fact \ref{F:Baudisch}, since $\B$ has Morley rank $2$. Also, it follows that $\z(\B) = [z,\B]$. 
Set $b=((0,z),(0,0))$ and $d=((0,0),(0,z))$, and let 
\[
f_1:S\to S\times S , \ x\mapsto f_1(x)= (\rho_b(x),\rho_d(x)). 
\] 
The above equation yields that 
\[
\mathrm{im}(\rho_b) = (\{0\} \times \{0\}) \times (\{0\} \times [\B,z]) = (\{0\} \times \{0\}) \times (\{0\} \times \z(\B)) =\A.
\]
Likewise, for $\mathrm{im}(\rho_d)$. In particular, observe that this yields that $\A$ is definable. On the other hand, the kernel is
\[
S_1 = \left\{ ( (x_1,x_2),(x_3,x_4) )\in S \ | \ [x_1,z]_\B = 0 = [x_3,z]_\B  \right\} = \big( (\z(\B)\times \B) \times (\z(\B)\times \B)  \big).
\] 
Thus, we get that $S/S_1$ is internal to $\A$. 

Now, set $a=((z,0),(0,0))$ and $c=((0,0),(z,0))$, and consider the additive homomorphisms \[
f_2:S_1\to S\times S , \ x\mapsto f_2(x)= (\rho_a(x),\rho_c(x)). 
\] 
Since for $x=((x_1,x_2),(x_3,x_4))\in S_1$ we have that $x_1,x_3\in \z(\B)$, we obtain that the images of $\rho_a$ and $\rho_c$ are $\A$. Furthermore, the kernel of $f_2$ is  
\begin{align*}
S_2 & = \left\{ ( (x_1,x_2),(x_3,x_4) )\in S_1 \ | \ [x_2,z]_\B = 0 = [x_4,z]_\B  \right\} \\ & = \big( (\z(\B)\times \z(\B)) \times (\z(\B)\times \z(\B))  \big) = \ann(S).
\end{align*}
Therefore, it follows that $S_1/S_2$ is internal to $\A$. Since $\ann(S) =S_2$, we deduce that $S/\ann(S)$ is analysable to $\A$, as desired. 

\noindent The second part of the statement follows as in Lemma \ref{L:Analysis}. Indeed, we know by  Proposition \ref{P:Ann} that $R^\circ/\ann(R^\circ)$ is almost $\Pp$-internal. By the above we also have that it is analysable to $\A$. So, Fact \ref{F:Internal} yields that $R^\circ/\ann(R^\circ)$ is almost internal to $\A$. This yields the statement.  	
\end{proof}
Now we state and prove the following:

\begin{theorem}\label{T:CM-trivialRing}
Let $\B$ be Baudisch's Lie $\F_q$-algebra.	The theory of the commutative unitary local ring  $\Lambda(\F_q,S(\B\ltimes \B^+))$ with finite residue field has finite Morley rank, is CM-trivial and does not satisfy the CBP. 
\end{theorem}
 \begin{proof} Set $\g=\B\ltimes \B^+$, $S= S(\g)$ and $R=\Lambda(\F_q,S)$. As $\B$ is a saturated model of its theory, so is $R$ since it is definable in the theory of $\B$. Furthermore, it has finite Morley rank and it is CM-trivial by N\"ubling's result \cite{hN05}.
  	
 We need only prove that $R$ does not satisfy the CBP.  So, suppose to get a contradiction that the theory of $R$ satisfies it and thus $R/\ann(R^\circ)$ is almost $\A$-internal by Lemma \ref{L:Analysis2}, where 
 \[
 \A=(\{0\}  \times \{0\} ) \times (\{0\}\times \z(\B)) \subset \{0_\g\} \times \g \subset S.
 \] 
 Let $A_0\subset R$ be a finite subset witnessing this. Since every ring automorphism of $R$ fixes pointwise $\F_q\subset R$, we may assume that $A_0\subset S$. Regard $S$ as the Cartesian product $\g\times \g$ and let $A\subset \g$ be a set containing all element appearing in a pair of $A_0$.  It follows from the proof of Theorem \ref{T:Lie_B} that $\g/\z(\g)$ is not almost internal to the definable subgroup  $\{0\}\times \B_2 = \{0\}\times \z(\B)$. Hence, for some element $a\in \g$ there are infinitely many automorphisms $\sigma_i\in \Aut(\g)$, for $i\in \N$, fixing pointwise $A\cup \left(\{0\}\times \z(\B) \right)$ such that $\sigma_i(a)- \sigma_j(a)\not\in \z(\g)$ for every distinct $i,j\in \N$.
 
 Now, consider the element $(a,0)\in \g\times\g =S$ and for each natural number $i$ let $\tau_i\in\Aut(R)$ be the automorphism defined by
 \[
 \tau_i \big( \alpha +  (c,d) ) \big) = \alpha +  (\sigma_i(c), \sigma_i(d)).
 \]
 It is clear that each $\tau_i$ is an automorphism of $R$ which fixes pointwise $A_0$ and $\A$, since $\sigma_i$ fixes $A\cup (\{0\}\times \z(\B))$. Also, for $i<j$ we have that
 \[
 \tau_i\big( (a,0) \big) - \tau_j\big( (a,0) \big)  = (\sigma_i(a)-\sigma_j(a),0) \not\in \z(\g) \times \z(\g) = \ann(S) = \ann(R^\circ).
 \]
 This shows that $R/\ann(R^\circ)$ is not almost internal to $\A$, which is a contradiction.
 \end{proof}

\end{document}